\documentclass[12pt,a4paper]{article}
\usepackage{amsmath}
\usepackage{amssymb}
\usepackage{amsthm}
\usepackage{amsfonts}
\usepackage{latexsym}

\theoremstyle{plain}
\newtheorem{thm}{Theorem}[section]
\newtheorem{cor}{Corollary}[section]
\newtheorem{lem}{Lemma}[section]
\newtheorem{prop}{Proposition}[section]

\newtheorem{claim}{Claim}[section]
\theoremstyle{remark}

\theoremstyle{definition}
\newtheorem{defn}{Definition}[section]

\newtheorem{exmp}{Example}[section]

\newtheorem{rem}{Remark}[section]

\title{Equivariant Asymptotics of Szeg\"{o} kernels under Hamiltonian $SU(2)$-actions}
\author{Andrea Galasso and Roberto Paoletti\footnote{\noindent{\bf Address:}
Dipartimento di Matematica e Applicazioni, Universit\`a degli Studi
di Milano Bicocca, Via R. Cozzi 55, 20125 Milano,
Italy; {\bf e-mail}: andrea.galasso@unimib.it, roberto.paoletti@unimib.it }}
\date{}

\begin{document}
\maketitle

\begin{abstract}
 Let $M$ be  complex projective manifold, and $A$ a positive line bundle on it. 
Assume that $SU(2)$ acts on $M$ in a Hamiltonian manner, with nowhere vanishing moment map,
and that this action linearizes
to $A$. Then there is an associated unitary representation of $G$ on the associated algebro-geometric
Hardy space, and the isotypical components are all finite dimensional. We consider the local and global asymptotic
properties the equivariant projector associated to a weight $k \, \boldsymbol{ \nu }$, when $\boldsymbol{ \nu }$ is fixed
and $k\rightarrow 
+\infty$.   
\end{abstract}

\section{Introduction}

Let $M$ be a connected complex $d$-dimensional projective manifold, 
and let $\omega$ be a  Hodge form on it. Thus $M$ has a natural choice
of a volume form, given by $\mathrm{d}V_M:=(1/d!)\,\omega^{\wedge d}$.

Suppose given, in addition, an action
$\mu :G\times M\rightarrow M$ of a compact and connected Lie group $G$,
which is holomorphic (meaning that each diffeomorphism 
$\mu_g:M\rightarrow M$, $g\in G$, is holomorphic), and Hamiltonian with respect
to $2\,\omega$, with moment map
$\Phi:M\rightarrow \mathfrak{g}^\vee$ ($\mathfrak{g}$ being of course the Lie algebra of $G$).

Let $(A,h)$ be a positive line bundle on $M$, such that unique compatible connection
on $A$ has curvature form $-2\pi\,\imath\,\omega$; let $A^\vee$ be the dual line bundle, 
and $X\subset A^\vee$ the unit circle bundle, with projection $p:X\rightarrow M$. Then $X$ is naturally a contact and CR manifold
by positivity of $A$; if $\alpha$ is the contact form, $X$ inherits the volume form
$\mathrm{d}V_X:=(2\pi)^{-1}\,\alpha\wedge p^*(\mathrm{d}V_M)$.

Furthermore, $X$ has a natural Riemannian structure $g_X$. The latter is uniquely determined by the following conditions: 
1): the vector sub-bundles $\mathcal{V}(X/M):=\ker (\mathrm{d}\pi), \,\mathcal{H}(X/M):= \ker(\alpha) \subset TX$
are mutually orthogonal; 2): $p:X\rightarrow M$ is a Riemannian submersion; 3): $S^1$ (under the standard action)
acts on $X$ by isometries; 4): the fibers of $p$ have unit length.

We shall henceforth identify
densities and half-densities on $X$, and accordingly use the abridged notation $L^2(X)$
for the space of square summable half-densities on $X$.

The Hamiltonian action $\mu$ naturally induces an infinitesimal contact action
of $\mathfrak{g}$ on $X$ \cite{k}; explicitly, if $\xi\in \mathfrak{g}$ and $\xi_M$ is the
corresponding Hamiltonian vector field on $M$, then its contact lift $\xi_X$ is as follows.
Let $\upsilon^\sharp$ denote the horizontal lift on $X$ of a vector field $\upsilon$ on $M$,
and denote by $\partial_\theta$ the generator of the structure circle action on $X$. Then
\begin{equation}
 \label{eqn:contact vector field}
\xi_X:= \xi_M^\sharp-\langle \Phi_G\circ p, \xi\rangle \,\partial_\theta.
\end{equation}
We shall make the stronger assumption that $\mu$ lifts  
to a contact action
$\widetilde{\mu}:G\times X\rightarrow X$ lifting $\mu$, and 
preserving the CR structure (in other words, $\mu$ linearizes to a metric
preserving action on $A$). 

Under these assumptions, there is a naturally induced unitary representation of $G$ on the Hardy space
$H(X)\subset L^2(X)$; hence $H(X)$ can be equivariantly decomposed over the irreducible representations of $G$:
\begin{equation}
 \label{eqn:decomposition of H(X)}
H(X)=\bigoplus _{\boldsymbol{\nu}\in \widehat{G}} H(X)_{\boldsymbol{\nu}},
\end{equation}
where $\widehat{G}$ is the collection of all irreducible representations of $G$.
As is well-known, if $\Phi (m)\neq 0$ for every $m\in M$, then each isotypical component $H(X)_{\boldsymbol{\nu}}$
is finite dimensional (see e.g. \S 2 of \cite{pao-IJM}).

For example, suppose that $G=S^1$ and $\mu$ is trivial, with moment map $\Phi=1$. The irreducible representations of
$S^1$ are indexed by the integers $k\in \mathbb{Z}$, and the isotypical component $H(X)_k$ may
be naturally and unitarily identified with the space $H^0\left(M,A^{\otimes k}\right)$ of global holomorphic sections
of $A^{\otimes k}$. Hence the spaces $H(X)_{\boldsymbol{\nu}}$ in (\ref{eqn:decomposition of H(X)}) 
may be viewed as an analogue of the usual algebro-geometric
notion of linear series, although they may \textit{not}, in general, 
be interpreted as spaces of sections of some power of $A$. It is then natural to study their global and local
properties and their geometric consequences, in analogy with the classical case.
Here the paradigm is offered by the TYZ asymptotic expansion and its near-diagonal extension; we
shall work in the general conceptual framework of \cite{boutet-guillemin} and \cite{guillemin-sternberg hq}, 
and adopt more specifically the approach developed in
\cite{z}, \cite{bsz} and \cite{sz}, which is based on the description of
$\Pi$ as an FIO in \cite{bs}.

Hence, as discussed in the introductions of \cite{pao-IJM} and \cite{gp}, the present perspective departs
form the setting of Berezin-Toeplitz quantization; rather, it may be considered a variant of it,
where the structure $S^1$-action on $X$ is replaced by the contact lift of a general Hamiltonian action 
of a compact Lie group on $M$ (see for instance \cite{charles}, \cite{mm}, \cite{mz},
\cite{schlichenmaier} and references therein for a discussion of Berezin-Toeplitz quantization
and different approaches to near-diagonal kernel asymptotics).

Let $\Pi_{\boldsymbol{\nu}}:L^2(X)\rightarrow H(X)_{\boldsymbol{\nu}}$ be the
orthogonal projector (the $\boldsymbol{\nu}$-equivariant Szeg\"{o} projector); 
if $H(X)_{\boldsymbol{\nu}}$ is finite dimensional,
the corresponding distributional kernel is smooth, 
$\Pi_{\boldsymbol{\nu}}(\cdot,\cdot)\in \mathcal{C}^\infty(X\times X)$ (the $\boldsymbol{\nu}$-equivariant Szeg\"{o} kernel).
We are generally interested in the asymptotic properties of the kernels $\Pi_{\boldsymbol{\nu}}(\cdot,\cdot)$
when $\boldsymbol{\nu}$ tends to infinity in weight space, and in finding analogues of the TYZ-asymptotic expansion and of
the near-diagonal scaling asymptotics which have been the object of considerable attention in the standard case $G=S^1$, $\Phi=1$.
The case where $G$ is a torus has been considered in \cite{pao-IJM}, \cite{pao-loa}, \cite{camosso}; in \cite{gp}, we have focused on
the case $G=U(2)$.

Here, we shall consider the case $G=SU(2)$. We shall henceforth write $G=SU(2)$ and 
$\mathfrak{g}=\mathfrak{su}(2)$ (the Lie algebra of $2\times 2$ traceless skew-Hermitian matrices).

The irreducible representations of $G$ are indexed by the integers $\nu> 0$. To emphasize the
difference with the case of $S^1$ without burdening the notation, we shall label 
the representations by the pairs $\boldsymbol{\nu}=(\nu,0)$, and denote them by $V_{\boldsymbol{\nu}}$.
As is well-known (see for instance \cite{var}), $V_{\boldsymbol{\nu}}=\mathrm{Sym}^{\nu-1}\left(\mathbb{C}^2\right)$, and  
the restriction to the standard torus
$T\leqslant G$ of the corresponding character $\chi_{\boldsymbol{\nu}}$ is
\begin{eqnarray}
\label{eqn:character chinu}
 \chi_{\boldsymbol{\nu}}\left(
\begin{pmatrix}
 e^{\imath\,\vartheta} & 0 \\
0 & e^{-\imath\,\vartheta}
\end{pmatrix}
\right)
&=&\frac{e^{\imath\,\nu\,\vartheta}-e^{-\imath\,\nu\,\vartheta}}{e^{\imath\,\vartheta}-e^{-\imath\,\vartheta}}\\
&=&e^{\imath\,(\nu-1)\,\vartheta}+e^{\imath\,(\nu-3)\,\vartheta}+\ldots+e^{-\imath\,(\nu-1)\,\vartheta}.\nonumber
\end{eqnarray}

We shall fix $\boldsymbol{\nu}$, and consider the pointwise asymptotics of $\Pi_{k\boldsymbol{\nu}}(\cdot,\cdot)$ 
when $k\rightarrow +\infty$. To begin with, we shall show that $\Pi_{k\boldsymbol{\nu}}(x,y)$ is rapidly decreasing,
unless $y\rightarrow G\cdot x$ (the orbit of $x$) at a sufficiently fast pace.

\begin{thm}
\label{thm:rapid decrease}
 Let us fix $C,\,\epsilon>0$. Then, uniformly for $(x,y)\in X\times X$ satisfying
$$
\mathrm{dist}_X\big(x, G\cdot y\big) \ge C\, k^{\epsilon -1/2},
$$
we have 
$\Pi_{k\boldsymbol{\nu}}(x,y)=O\left(k^{-\infty}\right)$.
\end{thm}

Let us consider the asymptotics of $\Pi_{k\boldsymbol{\nu}}$ near a point $x\in X$. To build-up to our next Theorems,
we need to introduce some more terminology. 

If $x\in X$, we shall set $m_x:=p(x)$.

For $m\in M$, $\Phi_G(m)\in \mathfrak{g}$ is traceless skew-Hermitian $2\times 2$ matrix. 

\begin{defn} 
\label{defn:coset}
Suppose $m\in M$ and 
$\Phi_G(m)\neq 0$. Then
\begin{enumerate}
 \item $\lambda (m)>0$ will denote the (unique) positive eigenvalue of 
$-\imath\, \Phi_G(m)$; 
\item $h_m\,T\in G/T$ will denote the unique coset such that 
\begin{equation}
 \label{eqn:defn of hm}
\Phi_G(m) =\imath\, h_m
\begin{pmatrix}
 \lambda (m)& 0 \\
0 & -\lambda (m)
\end{pmatrix}\, h_m^{-1};
\end{equation}
\item we shall set, for $\nu\in \mathbb{N}$,
\begin{equation*}
 u_0(\nu,m) := \frac{\nu}{2 \, \lambda (m)}.
\end{equation*}
If $x\in X$, we shall generally write  
$u_0(\nu,x)$ for $u_0(\nu,m_x)$.
\end{enumerate}

\end{defn}

The assignments $\lambda:M\rightarrow (0,+\infty)$ and $m\in M\mapsto h_m\,T\in G/T$ are $\mathcal{C}^\infty$,
provided of course that $\Phi_G(m)\neq 0$ for every $m\in M$.

\begin{rem}
\label{rem:positive eigenvalue}
The positive eigenvalue $\lambda(m)$ has a symplectic interpretation, being
closely related to the moment map for the action
restricted to a suitable torus $T_m\leqslant G$. Let us set
$$
\beta:=
\begin{pmatrix}
 \imath & 0\\
0 &-\imath
\end{pmatrix};
$$
thus $\beta$ is the infinitesimal generator of the standard torus $T$, and 
therefore $\mathrm{Ad}_{h_m} (\beta)$ is the infinitesimal generator of the torus
$T_m:=C_{h_m}(T)$ (here $C_g(h):=g\,h\,g^{-1}$, for all $g,\,h\in G$).
Then for any $m\in M$ we have
\begin{equation}
 \label{eqn:key equality}
2\,\lambda (m) =\left\langle h_m^{-1} \,\Phi_G(m)\, h_m, \beta\right\rangle= 
\big\langle \Phi_G(m) , \mathrm{Ad}_{h_m} (\beta) \big\rangle.
\end{equation}
\end{rem}

Let us denote by $G_x\leqslant G$ the stabilizer of $x\in X$. 
By equivariance of $\Phi_G$, $G_x$ stabilizes $\Phi_G(m)$. 
By (\ref{eqn:defn of hm}), 
$G_x\subset h_{m_x}\, T\, h_{m_x}^{-1} $. 

In particular, 
if $\widetilde{\mu}$ is locally free at $x$, then 
$G_x$ is finite and Abelian. There exist in this case
$e^{\imath\vartheta_j}\in S^1$, $j=1,\ldots, N_x$, such that 
\begin{equation}
 \label{eqn:stabilizer Gx}
G_x=\left\{h_{m_x} \, 
\begin{pmatrix}
 e^{\imath\vartheta_j} & 0 \\
0 & e^{-\imath\vartheta_j}
\end{pmatrix}
\,h_{m_x}^{-1}\,:\,j=1,\ldots,N_x \right\}.
\end{equation}
We shall set, for every $j=1,\ldots,N_x$,
\begin{equation}
 \label{eqn:defn di t_j}
t_j:=\begin{pmatrix}
 e^{\imath\vartheta_j} & 0 \\
0 & e^{-\imath\vartheta_j}
\end{pmatrix},\quad
g_j:=h_{m_x} \, 
t_j
\,h_{m_x}^{-1}.
\end{equation}

\begin{defn}
 \label{defn:center}
Let $Z:=\{\pm I_2\}\leqslant G$ be the center of $G$, and set
$Z_x:=G_x\cap Z$. 
\end{defn}

We shall see that there is a contribution to the asymptotics of 
$\Pi_{k\boldsymbol{\nu}}$ near $x$ associated to each $g\in G_x$, and that the shape of the
contribution is different depending on whether $g\in Z_x$ or $g\in G_x\setminus Z_x$.

If $h\in G_x\setminus Z_x$,
then $h\neq h^{-1}$. Hence $G_x\setminus Z_x$ has even cardinality $b_x=2\,a_x$, and $G_x$ has cardinality
$b_x+h$, where $h=1$ or $2$.
Perhaps after renumbering, we can assume that 
\begin{equation}
 \label{eqn:G_xminusZ_x}
G_x\setminus Z_x=\left\{g_1,\ldots,g_{a_x},g_{a_x+1}=g_1^{-1},\ldots,g_{b_x}=g_{a_x}^{-1}\right\}
\end{equation}
(it may well be that $a_x=0$).

\begin{defn}
 \label{defn:fl}
If $\ell \in \mathbb{Z}$, let us define $f_\ell:T\rightarrow \mathbb{C}$
by letting 
$$
f_\ell :e^{\vartheta\,\beta}\in T\mapsto e^{\imath\,\ell\,\vartheta}\in \mathbb{C}^*.
$$
\end{defn}

Let us first consider the on-diagonal asymptotics of $\Pi_{k\boldsymbol{\nu}}\left(x, x\right)$,
assuming only that $\widetilde{\mu}$ is locally free at $x$.

\begin{defn}
 \label{defn:definition of Bxj}
If $z\in \mathbb{C}$, let us set 
$$
A(z) := \imath
\begin{pmatrix}
 0 & z \\
\overline{z} & 0
\end{pmatrix}\in \mathfrak{g}.
$$
Then, with notation as in (\ref{eqn:defn di t_j}), the $\mathbb{R}$-linear map 
$$
\eta_j: z\in \mathbb{C} \mapsto \big(\mathrm{Ad}_{t_j^{-1}} - \mathrm{id}_{\mathfrak{g}}\big)\big( A(z)\big)
\in \mathfrak{g}
$$
is injective. Therefore, since $\widetilde{\mu}$ is locally free at $x$, there is a positive definite $2\times 2$ matrix
$C(x;j)$ such that
$$
\big\|\mathrm{Ad}_{h_{m_x}}\big(\eta_j(z)\big)_X(x)\big\|^2
= \frac{1}{2}\cdot  Z^t\,C(x;j)\,Z \quad (z\in \mathbb{C})
$$
where $Z:=
\begin{pmatrix}
 a&b
\end{pmatrix}^t\in \mathbb{R}^2$
if $z=a+\imath\,b$. Let us define
$$
B(x;j):=C(x;j)+4\,\imath\, \sin(2\vartheta_j)\cdot \lambda(m_x)\,I_2.
$$

\end{defn}

Finally, let us denote by $V_3$ the area of the unit sphere $S^3\subset \mathbb{R}^4$, and set 
\begin{equation}
 \label{eqn:DG/T}
D_{G/T}:=2\pi/V_3.
\end{equation}
The geometric meaning of $D_{G/T}$ will be elucidated by Lemma \ref{lem:DG/T}.

We shall prove the following.

\begin{thm}
 \label{thm:G_xnotinZx}
Assume that $\widetilde{\mu}$ is locally free at $x$, and   
that $G_x\setminus Z_x=\left\{g_1,g_1^{-1}, \ldots, g_{a_x}, g_{a_x}^{-1}\right\}$. Then
as $k\rightarrow +\infty$ there is an asymptotic expansion
\begin{eqnarray*}
\Pi_{k\,\boldsymbol{\nu}}(x,x)
& \sim &\Pi_{k\,\boldsymbol{\nu}}(x,x)_{Z_x}+\Pi_{k\,\boldsymbol{\nu}}(x,x)_{G_x\setminus Z_x}, 
\end{eqnarray*}
where
\begin{eqnarray}
\label{eqn:Z_x term}
\Pi_{k\,\boldsymbol{\nu}}(x,x)_{Z_x}
& \sim &\frac{1}{2\,\lambda (m_x)} \cdot \left(\frac{\nu \,k }{2\,\pi\,\lambda (m_x)}\right)^d \\
&& \cdot 
\sum_{g\in G_x} f_{1-k\cdot \nu}(g)\cdot 
\left[1+ \sum_{j\ge 1}^{+\infty} k^{-j} B_{gj}(x)\right], \nonumber
\end{eqnarray}
and 
\begin{eqnarray}
\label{eqn:non Z_x term}
\Pi_{k\,\boldsymbol{\nu}}(x,x)_{G_x\setminus Z_x}
& \sim &4\,\pi\cdot D_{G/T}\cdot \left(\frac{\nu\,k}{2\,\pi\cdot \lambda(m_x)}\right)^d\\
&&\cdot 
\left[\sum_{j=1}^{a_x}\Re\left(\frac{\imath\,\sin(\vartheta_j)\cdot e^{-\imath k \nu\cdot \vartheta_j}}{\sqrt{\det \big(B(x;j)\big)}}\right)
+\sum_{l\ge 1} k^{-l}\,P_{jl}(m_x)\right],\nonumber
\end{eqnarray}
for appropriate $\mathcal{C}^\infty$ functions $B_{gj},\, P_{jl}:M\rightarrow \mathbb{R}$.
\end{thm}

Let us next consider the near-diagonal asymptotics of $\Pi_{k\boldsymbol{\nu}}$.
As usual, these are conveniently expressed in 
Heisenberg local coordinates (HLC's) on $X$ (the reader is referred 
to \cite{sz} for a precise definition
and a general discussion thereof), and involve an invariant $\psi_2$ that we shall recall below.
To simplify our treatment, we shall assume in this case that $G_x=Z_x$.

Thus, given $x\in X$, let us choose 
a system of Heisenberg local coordinates (HLC's) on $X$ centered at $x$.
Following \cite{sz}, we shall denote this by the additive expression
$x+\upsilon\in X$, where
$\upsilon =(\theta,\mathbf{v})\in \mathbb{R}\times \mathbb{R}^{2n}$, with $\theta\in (-\pi,\pi)$ and
$\mathbf{v}$ sufficiently small. 
Translation in the \lq angular\rq \, coordinate $\theta$ corresponds to rotation in a fixed fiber of the 
circle bundle $X\rightarrow M$: whenever both sides are defined, 
$$
x+(\theta+\vartheta,\mathbf{v})=r_{\theta}\big(x+(\vartheta,\mathbf{v})\big),
$$
where $r_{\theta}(y)$ is the standard action of $e^{\imath\,\theta}\in S^1$ on $y\in X$.
On the other hand, the curve $\gamma:\tau\mapsto x+(\theta,\tau\,\mathbf{v})$ 
is \lq horizontal\rq \, for $\tau=0$, that is, $\dot{\gamma}(0)\in \mathcal{H}_{r_\theta(x)}(X/M)$.
When $\theta=0$, we shall abridge $x+(0,\mathbf{v})$ to $x+\mathbf{v}$. 

HLC's come with built-in unitary isomorphisms $T_{m_x}M\cong \mathbb{C}^d$ and $T_xX\cong \mathbb{R}\times T_{m_x}X$;
with this in mind, we shall use the expression $x+(\theta,\mathbf{v})$ for $(\theta,\mathbf{v})\in T_xX$ or
$x+\mathbf{v}$ for $\mathbf{v}\in T_{m_x}M$, where $m_x=\pi(x)$.

The following invariant plays an ubiquitous role in the study of rescaled local asymptotics of equivariant Szeg\"{o} kernels.

\begin{defn}
If $m\in M$ and $\mathbf{v}_1,\,\mathbf{v}_2\in T_mM$, following \cite{sz} let us set 
 \begin{eqnarray*}
\psi_2(\mathbf{v}_1,\,\mathbf{v}_2)&:=&
-\imath\, \omega_{ m } (\mathbf{v}_1 , \mathbf{v}_2 ) 
-\frac{1}{2} \, \big\| \mathbf{v}_1 - \mathbf{v}_2\big\|_m^2,
\end{eqnarray*}
where $\omega_m:T_mM\times T_mM\rightarrow \mathbb{R}$ is the symplectic form, and
$\|\cdot\|_m:T_mM\rightarrow \mathbb{R}$ is the norm function. 
\end{defn}

\begin{thm}
\label{thm:rescaled asymptotics}
Let us assume that $x\in X$ and that $\widetilde{\mu}$ is locally free on $X$ in a neighborhood
of $x$. Let $G_x \leqslant G$ be the stabilizer subgroup of $x$, and suppose that
$G_x\leqslant Z$.
Let us a choose system
of Heisenberg local coordinates on $X$ centered at $x$. 
Let us fix $C>0$ and $\epsilon\in (0,1/6)$. Then, uniformly for
$\mathbf{v}_1, \, \mathbf{v}_2\in T_{m_x}M$
satisfying $\|\mathbf{v}_j\|\le C\,k^{\epsilon}$ ($j=1,2$), and belonging to a subspace of $T_{m_x}M$ transverse to the $G$-orbit
through $m_x$,
we have for $k\rightarrow +\infty$ an asymptotic expansion of the form
\begin{eqnarray*}
\lefteqn{\Pi_{k\boldsymbol{\nu}}\left(x+\frac{1}{\sqrt{k}} \, \mathbf{v}_1, x+\frac{1}{\sqrt{k}} \, \mathbf{v}_2\right)}\\
& \sim &  \frac{1}{2\,\lambda (m_x)} \cdot \left(\frac{\nu \,k }{2\,\pi\,\lambda (m_x)}\right)^d \cdot 
\sum_{g\in G_x} f_{1-k\cdot \nu}(g)\cdot e^{u_0(\nu,m_x)\cdot \psi_2 \left(\mathbf{v}_1 ^{(g)}, \mathbf{v}_2 \right) }\\
&& \cdot 
\left[1+ \sum_{j\ge 1}^{+\infty} k^{-j/2} A_{gj}(x;\mathbf{v}_1^{(g)}, \mathbf{v}_2)\right],
\end{eqnarray*}
where
$A_{gj}(x;\cdot, \cdot)$ is a polynomial of degree $\le 3\,j$ and parity $(-1)^j$.
\end{thm}

We can apply Theorems \ref{thm:G_xnotinZx} and \ref{thm:rescaled asymptotics}
to estimate the dimension of $H(X)_{k\boldsymbol{\nu}}$ when $k\rightarrow +\infty$.
Let us make this explicit in the case where $\widetilde{\mu}$ is generically free,
leaving the possible variants to the interested reader.

\begin{cor}
 \label{cor:dimension estimate}
Assume that $\widetilde{\mu}$ is generically free (that is, $G_x$ is trivial for the general
$x\in X$). Then 
\begin{equation}
\label{eqn:dimension estimate}
\lim_{k\rightarrow+\infty} \left[\left(\frac{\pi}{k\,\nu}\right)^d \cdot H(X)_{k\boldsymbol{\nu}}\right]= 
\int_M \,\mathrm{d}V_M(m)\left[\left(\frac{1 }{2\,\lambda (m)}\right)^{d+1}\right].
\end{equation}
\end{cor}

\section{Examples}
Given $Z\in \mathbb{C}^{d+1}\setminus\{0\}$, we shall denote by $[Z]\in \mathbb{P}^d$
its image in projective space. If $Z=(z_0,\cdots,z_d)$, then $[Z]=[z_0:\cdots :z_d]$.

To begin with, let us test our normalizations against the simplest case of the standard action of $G$ on $\mathbb{P}^1$.

\begin{exmp}
 \label{exmp:P1}
Let $\omega_{FS}$ denote the Fubini-Study form on $\mathbb{P}^1$.
The standard action of $G$ on $\mathbb{P}^1$, given by
$\mu_A([Z]):=[AZ]$, is Hamiltonian with respect to $2\,\omega_{FS}$, with 
nowhere vanishing moment
map 
\begin{equation}
\label{eqn:moment map P1}
 \Psi \big([z_0:z_1]\big) := \frac{\imath}{|z_0|^2+|z_1|^2} \, 
\begin{pmatrix}
 \frac{1}{2}\,\left(|z_0|^2-|z_1|^2\right) & z_0\cdot \overline{z}_1\\
\overline{z}_0 \cdot z_1 &  \frac{1}{2}\,\left(|z_1|^2-|z_0|^2\right)
\end{pmatrix}.
\end{equation}

Then $\lambda ([z_0:z_1])=1/2$ for any $[z_0:z_1] \in \mathbb{ P }^1$, 
and the contact action $\widetilde{\mu}$ on $S^3$ is free, since it may be identified with 
action of $SU(2)$ on itself by left translations. Furthermore,
$H_{k\,\boldsymbol{\nu}}(X)=H_{k\nu-1}(X)$, where the right hand side is the $(k\cdot\nu-1)$-th isotype 
for the $S^1$-action.
With $\nu=1$ the leading order term
of the expansion of Theorem \ref{thm:rescaled asymptotics}
is 
$$
\left(\frac{k }{\pi}\right)^d \cdot 
e^{\psi_2 (\mathbf{v}_1,\mathbf{v}_2) },
$$
in agreement with the standard off-diagonal scaling asymptotics for Szeg\"{o}
kernels on $\mathbb{P}^1$ (\cite{bsz}, \cite{sz}).

\end{exmp}

\begin{exmp}
 \label{exmp:P1timesP1}
Let us consider the diagonal action of $G$ on $\mathbb{P}^1\times \mathbb{P}^1$,
$$\mu_A\big([Z],[W]\big)=\big([AZ],[AW]\big).$$
For $r=1,2,\ldots$, consider the symplectic structure $\Omega_r:=\omega_{FS}\boxplus (r\,\omega_{FS})$
on $\mathbb{P}^1\times \mathbb{P}^1$. Then $\mu$ is Hamiltonian with respect to $2\,\Omega_r$, with moment map
$$
\Phi_r :\big([Z],[W]\big) \mapsto
\Psi\big([Z]\big)+ r\, \Psi\big([W]\big).
$$
If $r\ge 2$, then $\Phi_r$ is nowhere vanishing.

On the other hand, $\Omega_r$ is the normalized curvature of the positive line bundle
$A_r:=\mathcal{O}_{\mathbb{P}^1}(1)\boxtimes \mathcal{O}_{\mathbb{P}^1}(r)$. 
The unit circle bundle $X_r$ associated to $A_r$
is the image of $S^3\times S^3$ under the map 
$$(Z,W)\in S^3\times S^3\subset \mathbb{C}^2\times \mathbb{C}^2\mapsto Z\otimes W^{\otimes r}\in \mathbb{C}^{2(r+1)},$$
and the contact lift of $\mu$ is given by 
$$
\widetilde{\mu}_A\left(Z\otimes W^{\otimes r}\right)=(AZ)\otimes (AW)^{\otimes r}.
$$

Let us consider the stabilizer subgroup of $Z\otimes W^{\otimes r}$. 
We have 
$$\widetilde{\mu}_A\left(Z\otimes W^{\otimes r}\right)=Z\otimes W^{\otimes r}\quad\Leftrightarrow
\quad AZ=\lambda_1\,Z, \quad AW=\lambda_2\,W$$ 
for certain $\lambda_1,\lambda_2\in S^1$ with $\lambda_1\cdot \lambda_2^r=1$.

If $Z$ and $W$ are linearly dependent, then $\lambda_1=\lambda_2$ and $\lambda_1^{r+1}=1$. The stabilizer subgroup
of $Z\otimes W^{\otimes r}$ is therefore cyclic of order $r+1$. Otherwise, $(Z,W)$ is an eigenbasis of
$A$ and $\lambda_2=\lambda_1^{-1}$, $\lambda_1^{r-1}=1$. 
Hence, assuming that $Z\wedge W\neq 0$,
the stabilizer subgroup of $Z\otimes W^{\otimes r}$ is cyclic of order $r-1$ when $(Z,W)$ 
is an orthonormal basis of $\mathbb{C}^2$, and otherwise it is trivial when $r$ is even and $\{\pm I_2\}$
when $r$ is odd. 
Thus $\widetilde{\mu}$ is locally free for $r\ge 2$. Furthermore, the action is generically free when $r$ is even,
and the stabilizer is generically of order two when $r$ is odd.

Let us now consider how $V_{k\,\boldsymbol{\nu}}$ appears in 
$$
H(X_r)= \bigoplus_{l=0}^{+\infty}H_l(X_r),\quad H_l(X_r)\cong
H^0\left(\mathbb{P}^1\times \mathbb{P}^1,A_r^{\otimes l}\right).
$$
Since $A_r^{\otimes l}=\mathcal{O}_{\mathbb{P}^1}(l)\boxtimes \mathcal{O}_{\mathbb{P}^1}(l\,r)$,
by the K\"{u}nneth formula we have
\begin{eqnarray*}
 H^0\left(\mathbb{P}^1\times \mathbb{P}^1,A_r^{\otimes l}\right)&\cong&
H^0\left(\mathbb{P}^1,\mathcal{O}_{\mathbb{P}^1}(l)\right)\otimes H^0\left(\mathbb{P}^1,\mathcal{O}_{\mathbb{P}^1}(l\,r)\right)\\
&\cong &V_{(l+1,0)}\otimes V_{(lr+1,0)}.
\end{eqnarray*}
Thus the character of $H^0\left(\mathbb{P}^1\times \mathbb{P}^1,A_r^{\otimes l}\right)$
as a $G$-representation is $\chi_{l+1}\cdot \chi_{l\,r+1}$. Using (\ref{eqn:character chinu}), we see by a few computations that 
\begin{eqnarray}
\label{eqn:product character}
\lefteqn{ (\chi_{l+1}\cdot \chi_{l\,r+1}) \left(
\begin{pmatrix}
 e^{\imath\,\vartheta} & 0 \\
0 & e^{-\imath\,\vartheta}
\end{pmatrix}
\right)}\\
&=&\left( e^{\imath\,l\,\theta}+e^{\imath\,(l-2)\,\theta}+\cdots+ e^{-\imath\,l\,\theta}\right)
\cdot \frac{e^{\imath\,(l\,r+1)\,\theta}-e^{-\imath\,(l\,r+1)\,\theta}}{e^{\imath\,\theta}-e^{-\imath\,\theta}}\nonumber\\
&=&\sum_{j=0}^l\frac{e^{\imath\,(l+l\,r+1-2j)\,\theta}-e^{-\imath\,(l+l\,r+1-2j)\,\theta}}{e^{\imath\,\theta}-e^{-\imath\,\theta}}.\nonumber
\end{eqnarray}
Therefore,
$$
H^0\left(\mathbb{P}^1\times \mathbb{P}^1,A_r^{\otimes l}\right)\cong 
\bigoplus_{j=0}^l V_{(l+l\,r+1-2j,0)}.
$$
We conclude that $V_{k\,\boldsymbol{\nu}}$ appears at most once in each $H_l(X_r)$; it does appear once, in fact, if and
only if $k\,\nu$ and $l\,(r+1)+1$ have the same parity, and 
\begin{equation}
 \label{eqn:admissible integers}
\frac{k\,\nu-1}{r-1} \ge l \ge \frac{k\,\nu-1}{r+1}.
\end{equation}
Suppose, for example, that $k\,\nu$ and $r+1$ are both even. Then $l\,(r+1)+1$ is odd for any choice of $l$ and we conclude that
$H_{k\,\boldsymbol{\nu}}(X)$ vanishes. Notice that at the general $x\in X_r$ we have $G_x=\{\pm I_2\}$,
and $\sum_{g\in G_x}f_{1-k\,\nu}(g)=0$. 
If, on the other hand, $r+1$ is even and $k\,\nu$ is odd, then there is a copy of 
$V_{k\,\boldsymbol{\nu}}$ in $H_l(X_r)$ for
every integer $l$ satisfying (\ref{eqn:admissible integers}). Hence the number of copies of 
$V_{k\,\boldsymbol{\nu}}$ in $H(X_r)$ is $\sim 2\,k\,\nu/\left(r^2-1\right)$, so that the dimension of
$H_{k\,\boldsymbol{\nu}}(X)$ is $\sim 2\,(k\,\nu)^2/\left(r^2-1\right)$. For the general $x\in X_r$, 
we have in this case $\sum_{g\in G_x}f_{1-k\,\nu}(g)=2$.

When $r+1$ is odd, on the other hand, the generic stabilizer is trivial. For the general $x\in X_r$, therefore,
$\sum_{g\in G_x}f_{1-k\,\nu}(g)=1$ irrespective of $k\,\nu$. 
If $k\,\nu$ is even (respectively, odd) then there is a copy of 
$V_{k\,\boldsymbol{\nu}}$ in $H_l(X_r)$ if and only if $l$ is odd (respectively, even) and satisfies (\ref{eqn:admissible integers}).
Thus the number of copies of 
$V_{k\,\boldsymbol{\nu}}$ in $H(X_r)$ is $\sim k\,\nu/\left(r^2-1\right)$, so that the dimension of
$H_{k\,\boldsymbol{\nu}}(X)$ is $\sim (k\,\nu)^2/\left(r^2-1\right)$.

\end{exmp}

\section{Preliminaries}

\label{sctn:preliminaries}


In this Section, we shall collect various basic concepts and foundational results that will be invoked in the
following proofs; in addition, in \S \ref{sctn:compact reduction} we shall establish a technical preamble to the proofs 
in \S \ref{sctn:proofs}.

For any $x\in X$ and $g\in G$,
$$
\Pi_{k\boldsymbol{\nu}}\left(\widetilde{\mu}_{g}(x),\widetilde{\mu}_{g}(x)\right)
=\Pi_{k\boldsymbol{\nu}} (x,x);
$$
furthermore, transplanting a system of HLC's centered at $x$ by $g\in G$ (in an obvious sense)
yields a system of HLC's centered at $\widetilde{\mu}_{g}(x)$. Therefore, with no loss of generality
we might replace $x$ by $\widetilde{\mu}_{h_{m_x}}(x)$ (recall (\ref{eqn:defn of hm})), and assume that
\begin{equation}
 \label{eqn:defn of hm1}
\Phi_G(m_x) =\imath\, 
\begin{pmatrix}
 \lambda (m_x)& 0 \\
0 & -\lambda (m_x)
\end{pmatrix}.
\end{equation}
Since however it is convenient to keep explicit track 
of $h_{m_x}$, we shall make the generic assumption that $\Phi_G(m_x)$ is not
anti-diagonal. 

\subsection{Recalls on Szeg\"{o} kernels}
\label{sctn:szego fio}

Let $\Pi:L^2(X)\rightarrow H(X)$ be the Szeg\"{o} projector, 
$\Pi(\cdot,\cdot)\in \mathcal{D}'(X\times X)$ the Szeg\"{o} kernel
(that is, the distributional kernel of $X$).
After \cite{bs} (see also the discussions in \cite{z}, \cite{bsz},
\cite{sz}), $\Pi$ is a FIO with complex phase, of the form
\begin{equation}
 \label{eqn:szegokernel}
\Pi(x,y)=\int_0^{+\infty}e^{iu\psi(x,y)}\,s(x,y,u)\,\mathrm{d}u,
\end{equation}
where $\Im (\psi)\ge 0$ and 
$$
s(x,y,u)\sim \sum_{j\ge 0}u^{d-j}\,s_j(x,y).
$$
We shall rely on the rather explicit description of $\psi$ in Heisenberg local coordinates
in \S 3 of \cite{sz}.
 
\subsection{The Weyl Integration Formula}
\label{sctn:weylint&char}

By composing $\Pi$ with the equivariant projector associated to $\boldsymbol{\mu}=(\mu>0)$
(see the discussion in \cite{guillemin-sternberg hq}),
we have 
\begin{equation}
 \label{eqn:equiv_projector}
\Pi_{\boldsymbol{\mu} }
\big (x',x'' \big)
=
\mu \cdot \int_G \, \mathrm{d} V_G (g) \left[\overline{\chi_{\boldsymbol{\mu}} (g)} 
\,\Pi\left( \widetilde{\mu}_{g^{-1}}( x' ),
x'' \right) \right].
\end{equation}
We can remanage (\ref{eqn:equiv_projector}) as follows. 
Let us define $F:T\rightarrow \mathcal{D}'(X\times X)$ by setting
\begin{equation}
 \label{eqn:defn of F(t)}
F(t;x',x''):=\int_{G/T} \, \mathrm{d} V_{G/T} (g\,T) \,
\left[
\Pi\left( \widetilde{\mu}_{g\,t^{-1}\,g^{-1}}( x' ),
x'' \right) \right]\quad (t\in T)\nonumber.
\end{equation}
Since $t$ and $t^{-1}$ are conjugate in $G$, $F(x',x'';t)=F\left(x',x'';t^{-1}\right)$.
Let $t_1$ and $t_2=t_1^{-1}$ denote the diagonal entries of 
$t\in T$. Then by the Weyl Integration and character formulae \cite{var}
\begin{eqnarray}
 \label{eqn:equiv_projector1}
\lefteqn{\Pi_{\boldsymbol{\mu} }
\big (x',x'' \big)
} \\
&=&
\frac{\mu}{2} \cdot \int_T\,\, \mathrm{d} V_{T} (t) \,
\left(t_1^{-\nu}-t_1^{\nu}\right)
\left(t_1-t_1^{-1}\right)\,F(t;x',x'')\nonumber\\
&=& I_+(\mu;x',x'')-I_-(\mu;x',x''), \nonumber
\end{eqnarray}
where 
\begin{eqnarray}
 \label{eqn:defn of I+}
I_{\pm}(\mu;x',x'')
:=\frac{\mu}{2} \cdot \int_T\,\, \mathrm{d} V_{T} (t) \,\left[t_1^{\mp \nu} \cdot\left(t_1-t_1^{-1}\right)\cdot
F(t;x',x'')\right].\nonumber
\end{eqnarray}
By (\ref{eqn:defn of F(t)}), the change of variable $t\mapsto t^{-1}$ shows that 
$I_-(\mu;x',x'')=-I_+(\mu;x',x'')$.
Hence,
\begin{eqnarray}
 \label{eqn:equiv_projector2}
\Pi_{\boldsymbol{\mu} }
\big (x',x'' \big)
&=&2\, I_+(\mu;x',x'')\\
&=&\mu \cdot \int_T\,\, \mathrm{d} V_{T} (t) \,\left[t_1^{- \nu} \cdot\left(t_1-t_1^{-1}\right)\cdot
F(t;x',x'')\right].\nonumber
\end{eqnarray}

\subsection{The Haar measure on $G/T$}

\label{sctn:local coordinates G/T}
As is well-known, $G$ is diffeomorphic to the unit sphere $S^3 \subset \mathbb{C}^2$ by the map 
\begin{equation}
 \label{eqn:S2 and S3}
\begin{pmatrix}
 \alpha & -\overline{ \beta } \\
\beta & \overline{ \alpha }
\end{pmatrix} \in G\stackrel{\gamma}{ \longrightarrow }
\begin{pmatrix}
 \alpha \\
\beta
\end{pmatrix}\in S^3.
\end{equation}
Furthermore, $\gamma$ intertwines the right action of $T\cong S^1$ 
on $G$ with the standard circle action on $S^3$.
Therefore, the projection $G\rightarrow G/T$ may be identified with the Hopf map 
$S^3\rightarrow \mathbb{P}^1 \cong S^2$. 
It follows that the Haar measure on $G/T$ is a positive multiple of 
the pull-back of the standard measure on $S^2$. 
Explicitly, using the local coordinates $(\theta, \delta)\in (0,\pi/2) \times (-\pi,\pi)$ 
for the coset in $G/T$
of the matrix (\ref{eqn:S2 and S3}) with $\alpha=\cos(\theta) \, e^{ \imath \delta }$,
$\beta = \sin (\theta)$, then the Haar volume element on $G/T$ is
$$
\mathrm{d}V_{G/T}(g\,T) = \frac{1}{2\pi} \,\sin( 2 \, \theta ) \, \mathrm{ d }\theta \,\mathrm{ d } \delta.
$$


\subsection{$G_x$-equivariant Heisenberg Local Coordinates}

Although inessential, it will slightly simplify our exposition to make a convenient choice of HLC's
centered at $x\in X$.
These depend on a system of \textit{preferred} adapted local coordinates at $m_x$,  
and of a \textit{preferred} local frame for $A$ at $m_x$ \cite{sz}. As to the former
(which needn't be holomorphic), 
we may use the exponential map at $m_x$, and for the latter we may assume without loss that 
it is $G_x$-invariant (by an argument as in \S 3 of \cite{pao-jsg0}). With this choice, we have the 
convenient equality
\begin{equation}
\label{eqn:invariant HLC}
 \widetilde{\mu}_g\big(x+(\theta,\mathbf{v})\big)=x+\big(\theta, \mathrm{d}_{m_x}\widetilde{\mu}_g(\mathbf{v})\big)
\quad (g\in G_x).
\end{equation}

\subsection{Reduction to compactly supported integrals}

\label{sctn:compact reduction}

For an arbitrary pair $(x_1,x_2)\in X\times X$, we consider the 
asymptotics of $\Pi_{k \boldsymbol{\nu} }
\big (x_{1} , x_{2} \big)$. Since 
\begin{equation}
\label{eqn:diagonal invariance}
 \Pi_{k \boldsymbol{\nu} }
\big (x_{1} , x_{2} \big)=
\Pi_{k \boldsymbol{\nu} }
\left (\widetilde{\mu}_g(x_{1}) , \widetilde{\mu}_g(x_{2}) \right)
\quad \forall\,g\in G, 
\end{equation}
we may assume without loss, by choosing $g\in G$ general, that $\Phi_G\circ \pi(x_2)$ is not anti-diagonal;
choosing $g=h_{m_{x_2}}$ (Definition \ref{defn:coset}), we may even reduce to the case where
$\Phi_G(m_{x_2})$ is diagonal.

By (\ref{eqn:equiv_projector}) with $\mu=k\,\nu$,
\begin{eqnarray}
\label{eqn:szego rescaled proj}
\Pi_{k \boldsymbol{\nu} }
\big (x_{1} , x_{2} \big)
& = &
k\nu \, \int_G \, \mathrm{d} V_G (g) \left[\overline{\chi_{k\boldsymbol{\nu}} (g)} \,\Pi\left( \widetilde{\mu}_{g^{-1}}( x_{1} ),
x_{2} \right) \right].
\end{eqnarray}

For some suitably small $\delta>0$, let us define 
\begin{eqnarray}
 \label{eqn:Gdelta}
G_{<\delta}(x_1,x_2)&:=&\big\{g\in G\,:\,\mathrm{dist}_X\left(\widetilde{\mu}_{g^{-1}}(x_1),x_2\right)<\delta\big\},\\
G_{>\delta}(x_1,x_2)&:=&\big\{g\in G\,:\,\mathrm{dist}_X\left(\widetilde{\mu}_{g^{-1}}(x_1),x_2\right)>\delta\big\}.
\nonumber
\end{eqnarray}

Then $\mathcal{U}:=\{G_{<2\,\delta}(x,y),\,G_{>\delta}(x,y)\}$ is an open cover of $G$, and we may consider 
a $\mathcal{C}^\infty$ partition of unity $\{\varrho, 1-\varrho\}$ of $G$ subordinate to $\mathcal{U}$.
One can see that $\varrho=\varrho_{x_1,x_2}$ may be chosen to depend smoothly on $(x_1,x_2)\in X\times X$;
we shall omit the dependence on $(x_1,x_2)$.

When $\varrho(g)\neq 1$, we have 
$
\mathrm{dist}_X\left( \widetilde{\mu}_{g^{-1}}( x_{1} ),
x_{2}\right) \ge \delta>0$.
Because $\Pi$ is smoothing away from the diagonal, the function 
$$
g\mapsto \big( 1-\varrho(g)\big)\cdot \Pi\left( \widetilde{\mu}_{g^{-1}}( x_{1} ),
x_{2} \right)
$$
is $\mathcal{C}^\infty$ on $G$. Therefore, taking Fourier transforms and arguing as 
in \S \ref{sctn:weylint&char}, we obtain the following Proposition.

\begin{prop}
 \label{prop:localization near G_x}
Only a rapidly decreasing contribution to the asymptotic is lost, if the integrand of 
(\ref{eqn:szego rescaled proj}) is multiplied by $\varrho(g)$. 
\end{prop}

On the support of $\varrho$, $\left( \widetilde{\mu}_{g^{-1}}( x_{1} ),
x_{2}\right)$ lies in a small neighborhood of the diagonal; since any smoothing term will
contribute negligibly to the asymptotics, we may replace $\Pi$ by its representation as
an FIO (\S \ref{sctn:szego fio}). 
If we insert (\ref{eqn:szegokernel}) in (\ref{eqn:szego rescaled proj})
(with the factor $\varrho (g)$ included), and apply the rescaling $u\mapsto k\,u$ we obtain
\begin{eqnarray}
\label{eqn:szego rescaled proj 1}
\Pi_{k \boldsymbol{\nu} }
\big (x_{1} , x_{2} \big)
& \sim &
k^2\nu \, \int_G \, \mathrm{d} V_G (g) \,\int_0^{+\infty}\,\mathrm{d}u\\
&&\left[\varrho(g)\cdot\overline{\chi_{k\boldsymbol{\nu}} (g)} \,
e^{
\imath\,k\,u\,\psi\left( \widetilde{\mu}_{g^{-1}}( x_{1} ),
x_{2} \right)
} \cdot s \left( \widetilde{\mu}_{g^{-1}}( x_{1} ),
x_{2}, k\,u \right) \right].\nonumber
\end{eqnarray}

Integration in
(\ref{eqn:szego rescaled proj 1}) can be reduced to a suitable compact domain
without altering the asymptotics.

\begin{prop}
 \label{prop:compact support u}
Let $D\gg 0$ and let $\rho \in \mathcal{C}^\infty_c(\mathbb{R})$ be $\ge 0$, supported in $\big(1/D, D\big)$,
and $\equiv 1$ on $(2/D, D/2)$.
Then only a rapidly decreasing contribution to the asymptotics is lost, 
if the integrand on the last line of (\ref{eqn:szego rescaled proj 1}) is
multiplied by $\rho(u)$. 
\end{prop}

\begin{proof}[Proof of Proposition \ref{prop:compact support u}]
Let us deal with the cases $u\gg 0$ and $0<u\ll 1$ separately.

\textbf{Case 1: $u\gg 0$.}

To begin with, let $\rho_1':(0,+\infty)\rightarrow [0,+\infty)$
be $\mathcal{C}^\infty$, $\equiv 1$ on $(0,D/2)$ and $\equiv 0$ on
$(D,+\infty)$ \footnote{ Throughout this proof, the primes do not stand for derivatives.}.
Let us set $\rho_2'(u):= 1-\rho'_1(u)$. By (\ref{eqn:szego rescaled proj 1}),
\begin{equation}
 \label{eqn:Piknu rho1}
\Pi_{k \boldsymbol{\nu} }
\big (x_{1} , x_{2} \big)\sim \Pi_{k \boldsymbol{\nu} }
\big (x_{1} , x_{2} \big)_1+\Pi_{k \boldsymbol{\nu} }
\big (x_{1} , x_{2} \big)_2,
\end{equation}
where 
\begin{eqnarray}
\label{eqn:szego rescaled proj 1j0}
\Pi_{k \boldsymbol{\nu} }
\big (x_{1} , x_{2} \big)_j
& \sim &
k^2\nu \, \int_G \, \mathrm{d} V_G (g) \,\int_0^{+\infty}\,\mathrm{d}u\,\left[e^{
\imath\,k\,u\,\psi\left( \widetilde{\mu}_{g^{-1}}( x_{1} ),
x_{2} \right)
} \right.\nonumber\\
&&\left.\varrho(g)\cdot\rho_j'(u)\cdot \overline{\chi_{k\boldsymbol{\nu}} (g)} \,
\cdot s \left( \widetilde{\mu}_{g^{-1}}( x_{1} ),
x_{2}, k\,u \right) \right].
\end{eqnarray}
We need to show that $\Pi_{k \boldsymbol{\nu} }
\big (x_{1} , x_{2} \big)_2=O\left(k^{-\infty}\right)$.

Let us set 
\begin{eqnarray}
 \label{eqn:open cover of Gdelta}
G'&:=&\big\{g\in G\,:\,\mathrm{dist}_G\big(g,\{\pm I_2\}\big)< 2\,\delta\big\},\\
G''&:=&\big\{g\in G\,:\,\mathrm{dist}_G\big(g,\{\pm I_2\}\big)>\delta\big\}.\nonumber
\end{eqnarray}
Then $\{G',G''\}$ is also an open cover of $G$, and we may consider a $\mathcal{C}^\infty$ partition of unity
$\beta'+\beta''=1$ on $G$ subordinate to it. Let us set 
$$\varrho':=\varrho \cdot \beta',\quad \varrho'':=\varrho \cdot \beta''.$$
Then $\varrho =\varrho'+\varrho''$,
where $\varrho'$ is supported in a small neighborhood of
$\{\pm I_2\}$, and $\varrho''$ is supported away from $\{\pm I_2\}$.

Accordingly, we have
\begin{equation}
 \label{eqn:Pik2primeprime}
\Pi_{k \boldsymbol{\nu} }
\big (x_{1} , x_{2} \big)_2
=\Pi_{k \boldsymbol{\nu} }
\big (x_{1} , x_{2} \big)_2'+
\Pi_{k \boldsymbol{\nu} }
\big (x_{1} , x_{2} \big)_2'',
\end{equation}
where in the former (respectively, latter) summand $\varrho (g)$ has been replaced by $\varrho'(g)$
(respectively, $\varrho''(g)$).

We shall deal with the two summands in (\ref{eqn:Pik2primeprime}) separately.

\begin{lem}
 \label{lem:Pik2prime}
$\Pi_{k \boldsymbol{\nu} }
\big (x_{1} , x_{2} \big)_2'=O\left(k^{-\infty}\right)$ as $k\rightarrow +\infty$.
\end{lem}

\begin{rem}
\label{rem:Pik2prime}
 In the integration defining $\Pi_{k \boldsymbol{\nu} }
\big (x_{1} , x_{2} \big)_2'$, $\widetilde{\mu}_{g^{-1}}( x_{1} )$ is close to $x_{2}$,
$g$ is close to $\{\pm I_2\}$, and $u$ is very large.
\end{rem}

\begin{proof}[Proof of Lemma \ref{lem:Pik2prime}]
Let us assume to fix ideas that $k\,\nu=2 \ell+1$ is odd.
Then $V_{k\boldsymbol{\nu}}$ may be identified with the vector space $\mathbb{C}^{(2\ell)}[z_1,z_2]$ of complex 
homogeneous polynomials of degree $2\,\ell$ in two variables.
A natural basis of the latter is given by the monomials $P_\mu (z_1,z_2):=z_1^{\ell-\mu}\,z_2^{\ell+\mu}$, where 
$\mu\in \{-\ell, \ldots, 0, \ldots, \ell\}$. We shall accordingly denote the matrix elements of the representation 
$V_{k\boldsymbol{\nu}}$ by $\mathcal{M}^{(k\boldsymbol{\nu})}_{a,b}(g)$, where $g\in G$ and $a,b\in  \{-\ell, \ldots, 0, \ldots, \ell\}$.
Thus 
\begin{eqnarray}
\label{eqn:szego rescaled proj 1ja}
\Pi_{k \boldsymbol{\nu} }
\big (x_{1} , x_{2} \big)_2'
& \sim &
\sum_{a=-\ell}^{\ell}\Pi_{k \boldsymbol{\nu} }
\big (x_{1} , x_{2} \big)_{2,a}',
\end{eqnarray}
where 
\begin{eqnarray}
\label{eqn:szego rescaled proj 1jb}
\Pi_{k \boldsymbol{\nu} }
\big (x_{1} , x_{2} \big)_{2,a}'
& := &
k^2\nu \, \int_G \, \mathrm{d} V_G (g) \,\int_{D/2}^{+\infty}\,\mathrm{d}u\,\left[e^{
\imath\,k\,u\,\psi\left( \widetilde{\mu}_{g^{-1}}( x_{1} ),
x_{2} \right)
} \right.\\
&&\left.\cdot \overline{\mathcal{M}^{(k\boldsymbol{\nu})}_{a,a}(g)} \cdot\varrho'(g)\cdot\rho_2'(u)
\cdot s \left( \widetilde{\mu}_{g^{-1}}( x_{1} ),
x_{2}, k\,u \right) \right].\nonumber
\end{eqnarray}

On the support of $\varrho'$, either $g\sim I_2$ or $g\sim -I_2$; hence we can write 
\begin{equation}
 \label{eqn:gthetadelta12}
g=
\begin{pmatrix}
 A(g)\,e^{\imath \theta_G (g)} &- \overline{\gamma(g)}\\
\gamma(g) &A(g)\,e^{-\imath \theta_G(g)} 
\end{pmatrix},
\end{equation}
where $A(g)>0$, and either $\theta_G(g)\sim 0$ or $\theta_G(g)\sim \pi$. 
Furthermore, $A,\,\gamma,\,\theta_G$ are $\mathcal{C}^\infty$
functions of $g\in G$ on a neighborhood of the support of $\rho'$.

By the discussion in \S 2.6.3 of
\cite{apple} and in \S 11 of \cite{rt}, with $g$ as in 
(\ref{eqn:gthetadelta12}) we have 
\begin{equation}
\label{eqn:diagonal matrix element}
\mathcal{M}^{(k\boldsymbol{\nu})}_{a,a}(g)=e^{-2\imath a \theta_G (g)} \cdot R_{\ell,a}\left(A(g)^2\right),
\end{equation}
where $R_{\ell,a}$ may be expressed in terms of suitable Jacobi polynomials, and is itself a real polynomial. 
Since the left hand side of (\ref{eqn:diagonal matrix element}) is an entry of a unitary matrix, we have at any rate
$\left|R_{\ell,a}\left(A(g)^2\right)\right|\le 1$. 

Inserting (\ref{eqn:diagonal matrix element}) in (\ref{eqn:szego rescaled proj 1jb}), we obtain:
\begin{eqnarray}
\label{eqn:szego rescaled proj 1jJacobi}
\Pi_{k \boldsymbol{\nu} }
\big (x_{1} , x_{2} \big)_{2,a}'
& := &
k^2\nu \, \int_G \, \mathrm{d} V_G (g) \,\int_{D/2}^{+\infty}\,\mathrm{d}u\,\left[e^{
\imath\,k\,\Psi_{a/k}(x_{1},x_2;u,g)
} \right.\\
&&\left.\cdot R_{\ell,a}\left(A(g)^2\right) \cdot\varrho'(g)\cdot\rho_2'(u)
\cdot s \left( \widetilde{\mu}_{g^{-1}}( x_{1} ),
x_{2}, k\,u \right) \right].\nonumber
\end{eqnarray}
where for every $b\in \mathbb{R}$ we set
\begin{equation}
\Psi_{b}(x_{1},x_2;u,g) := u\cdot\psi\left( \widetilde{\mu}_{g^{-1}}( x_{1} ),
x_{2} \right)+2\,b\cdot\theta_G (g).
\end{equation}
Since $|a/k|\le \nu/2$, the phases $\Psi_{a/k}$ form a bounded family.

Let $E:\mathfrak{\eta}\in \mathfrak{g}\mapsto E(\eta):=e^\eta\in G$
be the exponential map, and let $\beta$ be as in Remark \ref{rem:positive eigenvalue}.
Then every element of the standard torus $T\leqslant G$ may be written
\begin{equation}
 \label{eqn:t as exponential}
t=E(\theta \,\beta)=e^{\theta \beta}=\begin{pmatrix}
 e^{\imath \theta} & 0 \\
0 & e^{-\imath \theta}
\end{pmatrix}
.
\end{equation}

Let us view $\beta$ as a left-invariant vector field on $G$; the corresponding 1-parameter group of
diffeomorphisms is $\varphi_\tau(g) := g\,e^{\tau \,\beta}$. Thus
\begin{equation}
 \label{eqn:unitary flow}
\varphi_\tau : \begin{pmatrix}
 u & -\overline{v}\\
v & \overline{ u }
\end{pmatrix}  \mapsto 
\begin{pmatrix}
 u\cdot e^{\imath \tau} & -\overline{v}\cdot e^{-\imath \tau}\\
v\cdot e^{\imath \tau} & \overline{ u }\cdot e^{-\imath \tau}
\end{pmatrix}.
\end{equation}

Therefore, if $L_\beta$ is the same vector field viewed as a differential operator on $G$,
then $L_\beta(\theta_G)=1$ on the support of $\varrho'$. Also, $L_\beta$ is a skew-hermitian
operator on $L^2(G)$, since $\phi_\tau$ induces a 1-parameter group of unitary automorphisms of
$L^2(G)$; hence, $L^t_\beta=-\overline{L}_\beta$. 
Furthermore, the function $g\mapsto A(g)^2$ is in $\mathcal{C}^\infty (G)$, real and $\varphi_\tau$-invariant
for every $\tau\in \mathbb{R}$; therefore, 
$L_\beta\left( A(g)^2 \right)=\overline{L}_\beta\left(A(g)^2\right)=0$.

On the other hand, 
by (\ref{eqn:contact vector field}) we have 
\begin{eqnarray}
 \label{eqn:derivative of acted point}
\left.\frac{\mathrm{d}}{\mathrm{d}\tau}\widetilde{\mu}_{\varphi_\tau(g)^{-1}}(x_{1})\right|_{\tau=0}& = &
\left.\frac{\mathrm{d}}{\mathrm{d}\tau}\widetilde{\mu}_{e^{-\tau\,\beta}}\left(\widetilde{\mu}_{g^{-1}}(x_{1})\right)\right|_{\tau=0}
\\
&=& -\beta_X\left(\widetilde{\mu}_{g^{-1}}(x_{1})\right)\nonumber\\
&=&-\beta_M\left(\widetilde{\mu}_{g^{-1}}(m_{x_{1}})\right)^\sharp + \langle\Phi_G\left(\mu_{g^{-1}}(m_{x_{1}})\right),\beta\rangle \,
\partial_\theta.\nonumber
\end{eqnarray}

For $\varrho(g)\neq 0$ we have 
$\mathrm{dist}_X\left(\widetilde{\mu}_{g^{-1}}(x_{1}),x_2\right)\le 2\,\delta 
$;
therefore, 
\begin{equation}
 \label{eqn:moment map non zero}
\langle\Phi_G\left(\mu_{g^{-1}}(m_{x_{1}})\right),\beta\rangle =
\langle\Phi_G\left(m_{x_2}\right),\beta\rangle +O(\delta).
\end{equation}
If $\delta$ is sufficiently small, (\ref{eqn:moment map non zero}) is non-zero, since
we are assuming that $\Phi_G(m_{x_2})$ is not anti-diagonal;
assuming to fix ideas that $\Phi_G(m_{x_2})$ is diagonal,
then the right hand side of (\ref{eqn:moment map non zero}) is
$\langle\Phi_G\left(m_{x_2}\right),\beta\rangle=
2\,\lambda (m_x) +O(\delta)$.

In addition, by the discussion in \S \ref{sctn:szego fio}, where $\varrho(g)\neq 0$ 
\begin{equation}
 \label{eqn:diff psi near diagonal}
\mathrm{d}_{\left(\widetilde{\mu}_{g^{-1}}(x_{1}),x_{2}\right)}\psi = 
\left(\alpha_{\widetilde{\mu}_{g^{-1}}}(x_{1}),-\alpha_{x_{2}}\right)+O(\delta).
\end{equation}

Therefore,
\begin{eqnarray}
 \label{eqn:xi derivative of psi}
L_\beta \big(\Psi_{b}(x_{1k};u,g)\big) & = & 2\,\big[ u\cdot \lambda (m_x) + b\big] +O(\delta).
\end{eqnarray}
For $u\gg 0$, we conclude that 
$
L_\beta \big(\Psi_{a/k}(u,g)\big) \ge C'\cdot u +1
$
for some $C'>0$, which can be chosen uniformly for 
all $a\in \{-\ell,\ldots,0,\ldots,\ell\}$; 
by iteratively \lq integrating by parts\rq\, in (\ref{eqn:szego rescaled proj 1jJacobi}) 
by the transpose operator $L_\beta^t=-\overline{L}_\beta$,
we conclude that
$\Pi_{k \boldsymbol{\nu} }
\big (x_{1} , x_{2} \big)_{2,a}'=O\left(k^{-\infty}\right)$
uniformly for $a\in \{-\ell,\ldots,\ell\}$.
Since this holds uniformly for each of the $k\,\nu$ summands
in (\ref{eqn:szego rescaled proj 1ja}), the statement of Lemma \ref{lem:Pik2prime}
is established in the case where $k\,\nu$ is odd.

The case where $k\,\nu$ is even is only slightly different - one takes $\ell$ to be half-integer
(see Theorem 11.7.1 of \cite{rt}).
\end{proof}

\begin{lem}
 \label{lem:Pik2primeprime}
$\Pi_{k \boldsymbol{\nu} }
\big (x_{1} , x_{2} \big)_2''=O\left(k^{-\infty}\right)$ as $k\rightarrow +\infty$.
\end{lem}

\begin{rem}
\label{rem:Pik2primeprime}
 In the integration defining $\Pi_{k \boldsymbol{\nu} }
\big (x_{1} , x_{2} \big)_2''$, $\widetilde{\mu}_{g^{-1}}( x_{1} )$ is close to $x_{2}$,
$g$ is at a positive distance from $\{\pm I_2\}$, and $u$ is very large.
\end{rem}

\begin{proof}
 [Proof of Lemma \ref{lem:Pik2primeprime}]
The proof is similar to the one of Lemma \ref{lem:Pik2prime}, except that we shall use 
eigenvalues rather than matrix elements, so we'll be somewhat sketchy.

If $g\in G\setminus \{\pm I_2\}$, then there is a unique 
$\vartheta_G(g)=\cos^{-1}\big(\mathrm{trace}(g)/2\big)\in (0,\pi)$ such that
the eigenvalues of $g$ are $e^{\pm \imath\,\vartheta_G (g)}$.
The map $g\in G\setminus \{\pm I_2\}\mapsto \vartheta_G(g)\in (0,\pi)$ is $\mathcal{C}^\infty$.
On the same domain, the character of $V_{k\nu}$ is thus given by
\begin{equation}
 \label{eqn:character by eigenvalues}
\chi_{k\boldsymbol{\nu}} (g)= 
\sum_{j=0}^{k\,\nu-1}e^{\imath\,(k\nu-1-2j)\,\vartheta_G (g)}.
\end{equation}

In place of (\ref{eqn:szego rescaled proj 1ja}), we shall now write
\begin{eqnarray}
\label{eqn:szego rescaled proj 2ja}
\Pi_{k \boldsymbol{\nu} }
\big (x_{1} , x_{2} \big)_2''
& \sim &
\sum_{j=0}^{k\,\nu-1}\Pi_{k \boldsymbol{\nu} }
\big (x_{1} , x_{2} \big)_{2,j}'',
\end{eqnarray}
where 
\begin{eqnarray}
\label{eqn:szego rescaled proj 2jb}
\Pi_{k \boldsymbol{\nu} }
\big (x_{1} , x_{2} \big)_{2,j}''
& := &
k^2\nu \, \int_G \, \mathrm{d} V_G (g) \,\int_{D/2}^{+\infty}\,\mathrm{d}u\,\left[e^{
\imath\,k\,u\cdot \psi\left( \widetilde{\mu}_{g^{-1}}( x_{1} ),
x_{2} \right)
} \right.\\
&&\left.\cdot e^{-\imath\,(k\nu-1-2j)\cdot \vartheta_G (g)} \cdot\varrho''(g)\cdot\rho_2'(u)
\cdot s \left( \widetilde{\mu}_{g^{-1}}( x_{1} ),
x_{2}, k\,u \right) \right]\nonumber\\
& = & 
k^2\nu \, \int_G \, \mathrm{d} V_G (g) \,\int_{D/2}^{+\infty}\,\mathrm{d}u\,\left[e^{
\imath\,k\,\Gamma_{(1+2j)/k}(x_{1},x_2;u,g)} \right.\nonumber\\
&&\left.\cdot \varrho''(g)\cdot\rho_2'(u)
\cdot s \left( \widetilde{\mu}_{g^{-1}}( x_{1} ),
x_{2}, k\,u \right) \right], \nonumber
\end{eqnarray}
where we have set for $b\in \mathbb{R}$
\begin{equation}
 \label{eqn:Gamma jk}
\Gamma_{b}(x_{1},x_2;u,g) :=u\,\psi\left( \widetilde{\mu}_{g^{-1}}( x_{1} ),
x_{2} \right)-(\nu+b)\cdot \vartheta_G(g).
\end{equation}
Again, the phases $\Gamma_{(1+2j)/k}(x_{jk};\cdot,\cdot)$ vary in a bounded family.

Furthermore, since $\delta$ is small but fixed, $\vartheta_G$ is   
bounded in $\mathcal{C}^r$ norm for every $r\ge 0$ on the support of $\varrho''$.
We can then complete the proof by arguing as in the final part of the proof of Lemma 
\ref{lem:Pik2prime}.

\end{proof}

Given (\ref{eqn:Pik2primeprime}), Lemmata \ref{lem:Pik2prime} and \ref{lem:Pik2primeprime}
imply that for $
\Pi_{k \boldsymbol{\nu} }
\big (x_{1} , x_{2} \big)_2
=O\left(k^{-\infty}\right)$ for $k\rightarrow +\infty$.

\textbf{Case 2: $0<u\ll 1$.}

Let $\rho_1'':(0,+\infty)\rightarrow \mathbb{R}$ be $\mathcal{C}^\infty$, $\ge 0$, 
$\equiv 0$ on $(0,1/D)$ and $\equiv 1$ on $(2/D,+\infty)$. Let us set 
$\rho_2'':=1-\rho_2'$. 
By the above, we can replace (\ref{eqn:Pik2primeprime}) by 
\begin{equation}
 \label{eqn:Pik2small}
\Pi_{k \boldsymbol{\nu} }
\big (x_{1} , x_{2} \big)\sim \Pi_{k \boldsymbol{\nu} }
\big (x_{1} , x_{2} \big)_1=\Pi_{k \boldsymbol{\nu} }
\big (x_{1} , x_{2} \big)_{11}+
\Pi_{k \boldsymbol{\nu} }
\big (x_{1} , x_{2} \big)_{12},
\end{equation}
where $\Pi_{k \boldsymbol{\nu} }
\big (x_{1} , x_{2} \big)_{1j}$ is defined as in (\ref{eqn:szego rescaled proj 1j0}),
except that in the integrand $\rho_j'(u)$ is replaced by 
$\rho'_1(u) \cdot \rho''_j(u)$.

\begin{lem}
 \label{lem:Pik12prime}
$\Pi_{k \boldsymbol{\nu} }
\big (x_{1} , x_{2} \big)_{12}=O\left(k^{-\infty}\right)$ as $k\rightarrow +\infty$.
\end{lem}

\begin{proof}
[Proof of Lemma \ref{lem:Pik12prime}]
Let us rewrite $\Pi_{k \boldsymbol{\nu} }
\big (x_{1k} , x_{2k} \big)_{12}$ by means of the Weyl integration and character formulae,
as in \S \ref{sctn:weylint&char}.
Introducing coordinates on $T$ in (\ref{eqn:equiv_projector2}), we obtain
\begin{eqnarray}
 \label{eqn:Piknu 12 weyl}
\lefteqn{\Pi_{k \boldsymbol{\nu} }
\big (x_{1} , x_{2} \big)_{12}} \\
& = & \frac{k^2\cdot \nu}{2\,\pi} \cdot 
\int_{G/T}\,\mathrm{d}V_{G/T}(gT)\,\int_{-\pi}^{\pi}\,\mathrm{d}\vartheta \,
\left[e^{\imath\,k\,\Psi(x_{1},x_2;u,gT,\vartheta)}
\cdot \rho_2'(u)\right.\nonumber\\
&&\left.   \cdot\varrho''\left(g\,e^{-\vartheta\,\beta}\,g^{-1}\right)\cdot\left(e^{\imath\,\vartheta}-e^{-\imath\,\vartheta}\right)
\cdot s \left( \widetilde{\mu}_{g\,e^{-\vartheta\,\beta}\,g^{-1}}( x_{1} ),
x_{2}, k\,u \right)    \right],\nonumber
\end{eqnarray}
where we have set 
\begin{eqnarray}
 \label{eqn:phasePsijk}
\Psi(x_{1},x_2;u,gT,\vartheta):= u\cdot \psi\left( \widetilde{\mu}_{g\,e^{-\vartheta\,\beta}\,g^{-1}}( x_{1k} ),
x_{2k} \right)-
\nu\cdot \vartheta
\end{eqnarray}
Thus if $D\gg 0$ on the support of $\rho_2'(u)$ we have 
$$
\partial_\vartheta\Psi(x_{1},x_2;u,gT,\vartheta)\le -\frac{\nu}{2};
$$
the claim then follows in a standard manner by iteratively integrating by parts in $\mathrm{d}\vartheta$.
\end{proof}

We conclude that
$\Pi_{k \boldsymbol{\nu} }
\big (x_{1} , x_{2} \big)\sim \Pi_{k \boldsymbol{\nu} }
\big (x_{1} , x_{2} \big)_{11}$.
%
To complete the proof of Proposition \ref{prop:compact support u}, we need only factor
$\rho(u)=\rho'_1(u) \cdot \rho''_1(u)$.
\end{proof}

\begin{rem}
\label{rem:variant}
Let $\mathbf{v}_j\in T_{m_x}M$ be as in the statement of Theorem \ref{thm:rescaled asymptotics},
and set 
$
x_{jk}:=x+k^{-1/2}\, \mathbf{v}_j$.
 As a variant of Proposition \ref{prop:compact support u}, we can replace $(x_1,x_2)$
by $(x_{1k},x_{2k})$. The same arguments apply with minor modifications; in particular,
one will replace $G_{<\delta}(x_1,x_2)$ in (\ref{eqn:Gdelta}) by a $\delta$-neighborhood of the stabilizer 
subgroup $G_x$.
\end{rem}

\section{The proofs}
\label{sctn:proofs}
We collect in this Section the proofs of 
Theorems \ref{thm:rapid decrease}, \ref{thm:G_xnotinZx}, \ref{thm:rescaled asymptotics}.

\subsection{Theorem \ref{thm:rapid decrease}}
\label{sctn:proof rapid decrease}

\begin{proof}
 [Proof of Theorem \ref{thm:rapid decrease}]
We can replace $x$ and $y$ by any other points in their respective orbits, 
and therefore we may assume without loss that $\Phi_G\circ \pi(y)$ is diagonal, in form
(\ref{eqn:defn of hm}). 
By Proposition \ref{prop:compact support u} (with $x_1=x$ and $x_2=y$), we have
\begin{eqnarray}
\label{eqn:szego rescaled proj 1 for rapid decrease}
\Pi_{k \boldsymbol{\nu} }
\big (x , y \big)
& \sim &
k^2\nu \, \int_G \, \mathrm{d} V_G (g) \,\int_0^{+\infty}\,\mathrm{d}u\\
&&\left[\rho(u)\cdot \varrho(g)\cdot\overline{\chi_{k\boldsymbol{\nu}} (g)} \,
e^{
\imath\,k\,u\,\psi\left( \widetilde{\mu}_{g^{-1}}( x ),
y \right)
} \cdot s \left( \widetilde{\mu}_{g^{-1}}( x ),
y, k\,u \right) \right],\nonumber
\end{eqnarray}
where $\varrho\in \mathcal{C}^\infty(G)$ is a bump function supported where 
$\mathrm{dist}_X\left(\widetilde{\mu}_{g^{-1}}(x),y\right)\le 2\,\delta$,
and identically $\equiv 1$ where $\mathrm{dist}_X\left(\widetilde{\mu}_{g^{-1}}(x),y\right)\le \delta$,
while $\rho$ is as in Proposition \ref{prop:compact support u}.

Where $\mathrm{dist}_X(G\cdot x,y)\ge C\,k^{\epsilon-1/2}$, by Corollary 1.3 of \cite{bs} we have
\begin{equation}
 \label{eqn:bound on psi}
\left|\psi\left(\widetilde{\mu}_{g^{-1}}(x),y\right)\right|\ge \Im \left(\psi\left(\widetilde{\mu}_{g^{-1}}(x),y\right)\right|
\ge C_1\,k^{2\,\epsilon-1}
\end{equation}
for some constant $C_1>0$.

The statement of Theorem \ref{thm:rapid decrease} then follows by iteratively integrating by parts
in $\mathrm{d}u$, since at each step we introduce a factor $O\left(k^{-2\epsilon}\right)$.

\end{proof}

For expository reasons, we shall give the proof
of Theorem \ref{thm:rescaled asymptotics} before the one of Theorem \ref{thm:G_xnotinZx}.

\subsection{Theorem \ref{thm:rescaled asymptotics}}
\label{sctn:proof rescaled}

\begin{proof}
 [Proof of Theorem \ref{thm:rescaled asymptotics}]

By (\ref{eqn:diagonal invariance}),
we may assume without loss that $\Phi_G(m_x)$ is not antidiagonal.
Let $x_{jk}$ be as in Remark \ref{rem:variant}.
By the discussion in \S \ref{sctn:compact reduction},
\begin{eqnarray}
\label{eqn:Piknu and szego1}
\lefteqn{
\Pi_{k \boldsymbol{\nu} }
\big (x_{1k} , x_{2k} \big)
} \\
& \sim &
\frac{k^2\nu }{2\pi}\, \int_{1/D}^{D}\,\mathrm{d}u \,\int_{-\pi/2}^{3\pi/2} \,\mathrm{d}\vartheta \,\int_{G/T} \, \mathrm{d}V_{G/T} (gT)
\,\left[e^{\imath\, k\, \left[u\, \psi \left( \widetilde{\mu}_{g e^{-\imath \vartheta B} g^{-1}}( x_{1k} ),
x_{2k} \right) -\nu\, \vartheta\right]} \right.\nonumber\\
&&\left.\cdot \rho(u)\cdot \varrho \left(g e^{-\imath \vartheta B} g^{-1}\right)\cdot \left( e^{\imath\vartheta}-e^{-\imath\,\vartheta}\right)
\cdot
s\left( \widetilde{\mu}_{g e^{-\imath \vartheta B} g^{-1}}( x_{1k} ),
x_{2k}, k\, u \right)
\right] .\nonumber
\end{eqnarray}
The bump function $\varrho$ is supported near $G_x=Z_x\leqslant\{\pm I_2\}$ (Remark \ref{rem:variant}).
Hence, 
$\varrho=\varrho_+ +\varrho_ -$,
where $\varrho_+$ is supported in a small neighborhood of $I_2$, while $\varrho_-$
is supported in a small neighborhood of $-I_2$, and vanishes identically if $G_x$
is trivial. 
Writing $\varrho=\varrho_+ +\varrho_ -$ in (\ref{eqn:Piknu and szego1}), we obtain (with an obvious interpretation) 
\begin{equation}
 \label{eqn:Pik decomposed 0pi}
\Pi_{k \boldsymbol{\nu} }
\big (x_{1k} , x_{2k} \big) \sim \Pi_{k \boldsymbol{\nu} }
\big (x_{1k} , x_{2k} \big)_+ + \Pi_{k \boldsymbol{\nu} }
\big (x_{1k} , x_{2k} \big)_- ;
\end{equation}
let us examine the two summands in (\ref{eqn:Pik decomposed 0pi}) separately.

\subsubsection{The asymptotics of $\Pi_{k \boldsymbol{\nu} }
\big (x_{1k} , x_{2k} \big)_+$}

Integration in $\mathrm{d}V_G(g)$ is supported in a small neighborhood of $I_2$.

\begin{prop}
\label{prop:rescaled case 0}
Under the assumptions of Theorem \ref{thm:rescaled asymptotics},
as $k\rightarrow +\infty$ we have 
\begin{eqnarray*}
\Pi_{k\boldsymbol{\nu}}\left(x_{1k} , x_{2k}\right)_+
& \sim &  \frac{1}{2\,\lambda (m_x)} \cdot \left(\frac{\nu \,k }{2\,\pi\,\lambda (m_x)}\right)^d \cdot 
e^{u_0(\nu,m_x)\cdot \psi_2 \left(\mathbf{v}_1 , \mathbf{v}_2 \right) }\\
&& \cdot 
\left[1+ \sum_{j\ge 1}^{+\infty} k^{-j/2} A^+_j(x;\mathbf{v}_1, \mathbf{v}_2)\right],
\end{eqnarray*}
where
$A^+_j(x;\cdot, \cdot)$ is a polynomial of degree $\le 3\,j$ and parity $(-1)^j$.
\end{prop}

\begin{proof}
 [Proof of Proposition \ref{prop:rescaled case 0}]
$\Pi_{k \boldsymbol{\nu} }
\big (x_{1k} , x_{2k} \big)_+$ is given by (\ref{eqn:Piknu and szego1}), 
with the the cut-off $\varrho \left(g e^{-\imath \vartheta B} g^{-1}\right)$
replaced by $\varrho_+ \left(g e^{-\imath \vartheta B} g^{-1}\right)$; therefore,
integration in $\mathrm{d}\vartheta$ is restricted to $(-2\,\delta,2\,\delta)$; 
we may assume that $\varrho_+ \left(g e^{-\imath \vartheta B} g^{-1}\right)$
is identically equal to one for $\vartheta\in (-\delta,\delta)$. 

Let us fix constants $C_1>0$, $\epsilon_1\in (0,1/6)$. 
Iteratively integrating by parts in $\mathrm{d}u$, similarly to the proof 
of Theorem \ref{thm:rapid decrease}, we conclude that  
the locus where $|\vartheta|>C_1 \, k^{\epsilon_1-1/2}$ 
contributes negligibly to the asymptotics of $\Pi_{k \boldsymbol{\nu} }
\big (x_{1k} , x_{2k} \big)_0$. 
Hence we conclude the following.

\begin{lem}
 \label{lem:shrinking support vartheta}
Suppose that $\varrho_1\in \mathcal{C}_c(\mathbb{R})$ is $\ge 0$, supported in $\big(-2, 2\big)$,
and $\equiv 1$ on $(-1, 1)$. Then 
the asymptotics of $\Pi_{k \boldsymbol{\nu} }
\big (x_{1k} , x_{2k} \big)_+$ are unchanged, if the integrand is multiplied 
by $\varrho_1\left(k^{1/2-\epsilon_1}\,\vartheta\right)$. 
\end{lem}

Applying the rescaling $\vartheta \mapsto \vartheta/ \sqrt{k}$, we recover

\begin{eqnarray}
\label{eqn:Piknu and szego2}
\lefteqn{
\Pi_{k \boldsymbol{\nu} }
\big (x_{1k} , x_{2k} \big)_+
} \\
&\sim & \frac{k^{3/2}\,\nu }{2\pi}\, \int_0^{+\infty}\,\mathrm{d}u \,
\int_{-\infty}^{+\infty} \,\mathrm{d}\vartheta \,
\int_{G/T} \, \mathrm{d}V_{G/T} (gT)
\left[e^{\imath\, k\, \Psi^+_k} \cdot \varrho_1\left(k^{-\epsilon_1}\,\vartheta\right)\cdot \rho  ( u ) \right.
\nonumber\\
&&\left.\cdot
\left( e^{\imath\vartheta/\sqrt{k}}-e^{-\imath\,\vartheta/\sqrt{k}}\right)\cdot 
s\left( \widetilde{\mu}_{g e^{-\imath \vartheta B/\sqrt{k}} g^{-1}}( x_{1k} ),
x_{2k}, k\, u \right)
\right] \nonumber,
\end{eqnarray}
where
\begin{equation}
 \label{eqn:defn di Psik}
\Psi_k^+ (u, \mathbf{v}_1 , \mathbf{v}_2 ,\vartheta, g\,T)  := u\, \psi \left( \widetilde{\mu}_{g e^{-\imath \vartheta B/\sqrt{k}} g^{-1}}( x_{1k} ),
x_{2k} \right) -\frac{\vartheta}{\sqrt{k}}\, \nu.
\end{equation}
Integration in $\mathrm{d}\vartheta$ is over an interval of length $4\, k^{\epsilon_1}$ centered at the origin.

Let us make $\Psi_k$ more explicit. By Corollary 2.2 of \cite{pao-IJM}, with $m_x=\pi(x)$ we have
\begin{eqnarray}
 \label{eqn:HLCaction}
\lefteqn{\widetilde{\mu}_{g e^{-\imath \vartheta B/\sqrt{k}} g^{-1}}( x_{1k} )
= \widetilde{\mu}_{e^{-\vartheta \mathrm{Ad}_g( \beta)/\sqrt{k}} }( x_{1k} ) }\\
& = & 
x + \left ( \Theta_k ( \mathbf{v}_1 ,  \vartheta , g\,T), 
\frac{1}{\sqrt{k}} \,\mathrm{V}(\mathbf{v}_1,  \vartheta , gT)
+R_2 \left(\frac{1}{\sqrt{k}} \, \vartheta , \frac{1}{\sqrt{k}} \, \mathbf{v}_1\right)\right),\nonumber
\end{eqnarray}
where (for appropriate $R_3$ and $R_2$)
\begin{eqnarray}
 \label{eqn:defn of Thetak}
\Theta_k ( \upsilon_1 , \vartheta , g\,T) 
& : = &
\frac{1}{\sqrt{k}} \, 
\vartheta \cdot
\Big\langle\Phi_G (m_x),  \mathrm{Ad}_g( \beta) \Big\rangle 
\\
&& + \frac{1}{k}\, \vartheta\cdot \omega_{ m_x } \big( \mathrm{Ad}_g( \beta)_M (m), \mathbf{v}_1 \big)  
+R_3 \left(\frac{1}{\sqrt{k}} \, \vartheta , \frac{1}{\sqrt{k}} \, \mathbf{v}_1\right),
\nonumber
\end{eqnarray}
\begin{eqnarray}
 \label{eqn:defn of Vk}
\mathrm{V}(\mathbf{v}_1, \vartheta , gT)
 : = 
\mathbf{v}_1- \vartheta \,\mathrm{Ad}_g\big(\beta\big)_M(m_x).
\end{eqnarray}
We shall use the abridged notation $\Theta_k$ and $V$.

In abridged notation, let us set 
\begin{eqnarray}
 \label{eqn:defn of Thetaktilde}
\widetilde{\Theta}_k ( \upsilon_1 , \upsilon_2 , \vartheta , g\,T)
& : = & 
\frac{1}{\sqrt{k}} \, A+ \frac{1}{k}\, B
+R_3 \left(\frac{1}{\sqrt{k}} \, \vartheta , \frac{1}{\sqrt{k}} \, \mathbf{v}_1\right),
\nonumber
\end{eqnarray}
where
\begin{equation}
 \label{eqn:defn di A}
A = A ( \mathbf{v}_1 , \mathbf{v}_2 , \vartheta , g\,T)
:= 
\vartheta \cdot 
\Big\langle\Phi_G (m_x),  \mathrm{Ad}_g( \beta) \Big\rangle ,
\end{equation}
\begin{equation}
 \label{eqn:defn di B}
 B=B  ( \mathbf{v}_1 , \vartheta , g\,T) := \vartheta\cdot \omega_{ m_x } \big( \mathrm{Ad}_g( \beta)_M (m), \mathbf{v}_1 \big) .
\end{equation}
Then 
\begin{eqnarray}
 \label{eqn:Psi_kexpanded}
\Psi_k^+ (u, \mathbf{v}_1 , \mathbf{v}_2 ,\vartheta, g\,T) 
& = &   \imath \, u \,  \left[ 1 - e^{ \imath \, \widetilde{\Theta}_k}\right]  -\frac{\vartheta}{\sqrt{k}}\, \nu 
-\imath \, \frac{ u }{ k } \, \psi_2 \Big( V ,
\mathbf{v}_2\Big) \nonumber \\
&& +R_3\left(\frac{1}{\sqrt{k}} \, (\vartheta ,  \mathbf{v}_1, \mathbf{v}_2 )\right).
\end{eqnarray}
In view of \S 3 of \cite{sz} (see especially (65)), 
by a few computations we obtain the following.

\begin{lem}
 \label{lem:expansion of Psik}
We have
\begin{eqnarray*}
 \Psi_k ^+(u, \mathbf{v}_1 , \mathbf{v}_2 ,\vartheta, g\,T) & : = & 
\frac{1}{ \sqrt{k} } \, \mathcal{G} (u, \vartheta, g\,T)
+ \frac{ 1 }{ k } \, \mathcal{D}  (u, \mathbf{v}_1 , \mathbf{v}_2 ,\vartheta, g\,T)\\
&& + R_3 \left(\frac{1}{\sqrt{k}} \, (\vartheta ,  \mathbf{v}_1 , \mathbf{v}_2 )\right),
\end{eqnarray*}
where 
\begin{eqnarray*}
\mathcal{G} (u, \vartheta, g\,T)
& = & u \, A -\vartheta \, \nu 
\\
& = & 
\vartheta \cdot \left[ u \, 
 \big \langle \Phi_G ( m_x ), \mathrm{ Ad }_g ( \beta ) \big\rangle 
- \vartheta \right],
\end{eqnarray*}
\begin{eqnarray*}
\mathcal{D}  (u, \mathbf{v}_1 , \mathbf{v}_2 ,\vartheta, g\,T)
& = & u \left[ B + \imath \left(\frac{ 1 }{ 2 } \, A^2 - \psi_2 \big( V, \mathbf{v}_2 \big) \right) \right].
\end{eqnarray*}

\end{lem}

From (\ref{eqn:Piknu and szego2}) and Lemma \ref{lem:expansion of Psik}, we conclude that
\begin{eqnarray}
\label{eqn:Piknu and szego3}
\lefteqn{
\Pi_{k \boldsymbol{\nu} }
\big (x_{1k} , x_{2k} \big)_+
} \\
&\sim & \frac{k^{3/2}\,\nu }{2\pi}\, \int_0^{+\infty}\,\mathrm{d}u \,
\int_{-\infty}^{+\infty} \,\mathrm{d}\vartheta \,\int_{G/T} \, \mathrm{d}V_{G/T} (gT)
\nonumber\\
&&\left[e^{\imath\, \sqrt{k}\, \mathcal{G}(u, \vartheta, g\,T) } \cdot e^{\imath \, u \, B 
+u \, \big[ \psi_2 \big( V, \mathbf{v}_2 \big) -\frac{ 1 }{ 2 } \, A^2     \big]}\cdot \varrho_1\left(k^{-\epsilon_1}\,\vartheta\right) 
\cdot  e^{\imath\,k\cdot R_3\left(\frac{1}{\sqrt{k}} \, (\vartheta ,  \mathbf{v}_1, \mathbf{v}_2 )\right)} \right.
\nonumber \\
&& \left.\cdot \rho  ( u )  \cdot \left( e^{\imath\vartheta/\sqrt{k}}-e^{-\imath\,\vartheta/\sqrt{k}}\right)
\cdot s\left( \widetilde{\mu}_{g e^{-\imath \vartheta B/\sqrt{k}} g^{-1}}( x_{1k} ),
x_{2k}, k\, u \right)
\right] \nonumber,
\end{eqnarray}
with $A$ and $B$ as in (\ref{eqn:defn di A}) and (\ref{eqn:defn di B}).

We can make (\ref{eqn:Piknu and szego3}) yet more explicit, introducing coordinates 
$(\theta , \delta )$ on $G/T$ as in \S \ref{sctn:local coordinates G/T}. 
Furthermore, let $h_m \,T\in G/T$ be as in (\ref{eqn:defn of hm}), and 
operate the change of variable $g\,T \mapsto h_m \,g\,T$ in $G/T$;
we shall write $g$ in the form (\ref{eqn:S2 and S3}) with $\alpha =\cos (\theta) \, e^{\imath \delta}$
and $\beta=\sin(\theta)$. 
Then 
\begin{eqnarray}
 \label{eqn:Gamma riscritta}
\mathcal{G}(u, \vartheta, h_m g\,T)
 & = &\vartheta \cdot   \left[ 
u\cdot \left \langle \imath\,g^{-1}
\begin{pmatrix}
 \lambda (m_x) & 0 \\
0 & -\lambda (m_x)
\end{pmatrix}
\, g
,  \beta  \right\rangle - \nu  \right]
  \nonumber\\
& = & \vartheta \cdot \big[ 
u \cdot \cos(2\theta) \cdot 2 \,\lambda (m_x) - \nu \big]
.   
\end{eqnarray}

Let $\mathcal{G}'(u, \vartheta, \theta)$ denote the expression on the
last line of (\ref{eqn:Gamma riscritta}).
We can rewrite (\ref{eqn:Piknu and szego3}) in the following form:
\begin{eqnarray}
\label{eqn:Piknu and szego4}
\lefteqn{
\Pi_{k \boldsymbol{\nu} }
\big (x_{1k} , x_{2k} \big)_+
} \\
&\sim & \frac{k^{3/2}\,\nu }{4\pi^2}\, \int_0^{+\infty}\,\mathrm{d}u \,
\int_{-\infty}^{+\infty} \,\mathrm{d}\vartheta \,
\int_0^{\pi / 2}\,\mathrm{ d }\theta \int_{-\pi}^{\pi}\, \mathrm{d} \delta\, 
\nonumber\\
&&\left[e^{\imath\, \sqrt{k}\, \mathcal{G}'(u, \vartheta, \theta) } \cdot e^{\imath \, u \, B 
+u \, \big[ \psi_2 \big( V, \mathbf{v}_2 \big) -\frac{ 1 }{ 2 } \, A^2     \big]}\cdot \varrho_1\left(k^{-\epsilon_1}\,\vartheta\right) 
\cdot  e^{\imath\,k\cdot R_3\left(\frac{1}{\sqrt{k}} \, (\vartheta ,  \mathbf{v}_1, \mathbf{v}_2 )\right)} \right.
\nonumber \\
&& \left.\cdot \rho ( u ) \cdot \left( e^{\imath\vartheta/\sqrt{k}}-e^{-\imath\,\vartheta/\sqrt{k}}\right)
\cdot s\left( \widetilde{\mu}_{g e^{-\imath \vartheta B/\sqrt{k}} g^{-1}}( x_{1k} ),
x_{2k}, k\, u \right) \, \sin ( 2\theta)
\right] \nonumber,
\end{eqnarray}
where (with abuse of notation)
$g=g(\theta, \delta)$ and $A=A(\theta, \delta)$, $B=B(\theta,\delta)$ by the obvious change of variables.

Setting $t = \cos (2\theta)$, we can reformulate
(\ref{eqn:Piknu and szego4}) as follows. With slight abuse, let us write $g\,T=g(t,\delta)\,T$ and define
\begin{eqnarray}
 \label{eqn:defn di Gamma}
\Gamma(t; u, \vartheta) & := &
 \mathcal{G}'(u, \vartheta, \theta) 
=\vartheta \cdot \big[2 \,  \lambda (m_x) \cdot u \cdot
t  - \nu\big]
.
\end{eqnarray}
Then
\begin{eqnarray}
\label{eqn:Piknu and szego5}
\lefteqn{
\Pi_{k \boldsymbol{\nu} }
\big (x_{1k} , x_{2k} \big)_+
} \\
&\sim & \frac{1}{2}\cdot\frac{k^{3/2}\,\nu }{(2\,\pi)^2}\, \int_0^{+\infty}\,\mathrm{d}u \,
\int_{-\infty}^{+\infty} \,\mathrm{d}\vartheta \,
\int_{-1}^{ 1 }\,\mathrm{ d }t \int_{-\pi}^{\pi}\, \mathrm{d} \delta\, 
\nonumber\\
&&\left[e^{\imath\, \sqrt{k}\, \Gamma (t; u, \vartheta) } \cdot 
e^{\imath \, u \, B_t 
+u \, \big[ \psi_2 \big( V_t, \mathbf{v}_2 \big) -\frac{ 1 }{ 2 } \, A_t^2     \big]}
\cdot \varrho_1\left(k^{-\epsilon_1}\,\vartheta\right) 
\cdot  e^{\imath\,k\cdot R_3\left(\frac{1}{\sqrt{k}} \, (\vartheta ,  \mathbf{v}_1, \mathbf{v}_2 )\right)} \right.
\nonumber \\
&& \left.\cdot \rho  ( u ) \cdot \left( e^{\imath\vartheta/\sqrt{k}}-e^{-\imath\,\vartheta/\sqrt{k}}\right)
\cdot s\left( \widetilde{\mu}_{g e^{-\imath \vartheta B/\sqrt{k}} g^{-1}}( x_{1k} ),
x_{2k}, k\, u \right) \right]; \nonumber
\end{eqnarray}
we have denoted by $A_t$, $B_t$ the functions
\begin{equation*}
A_t\big( \mathbf{v}_1 , \mathbf{v}_2 , \vartheta,\delta) = A_t ( \mathbf{v}_1 , \mathbf{v}_2 , \vartheta, g(t,\delta)\,T \big),
\quad 
 B_t ( \mathbf{v}_1 , \vartheta , \delta)=B  ( \mathbf{v}_1 , \vartheta , g(t,\delta)\,T),
\end{equation*}
and similarly for $V_t$.

Let us remark that
\begin{eqnarray}
 \label{eqn:Delta term}
e^{\imath\vartheta/\sqrt{k}}-e^{-\imath\,\vartheta/\sqrt{k}} & = &
\frac{2\,\imath}{\sqrt{ k }}\cdot \vartheta \cdot 
\sum_{j=0}^{+\infty} \frac{ (-1) ^j}{(2j+1)!} \cdot \frac{ \vartheta ^{ 2j }}{ k^j } \\
& = & \frac{2\,\imath}{\sqrt{ k }}\cdot \vartheta \cdot 
\left [ 1 + R_2\left(\frac{\vartheta}{ \sqrt{k} } \right) \right]. \nonumber
\end{eqnarray}
Furthermore, working in HLC's, Taylor expansion yields an asymptotic expansion 
\begin{equation}
 \label{eqn:s term}
s\left( \widetilde{\mu}_{g e^{-\imath \vartheta B/\sqrt{k}} g^{-1}}( x_{1k} ),
x_{2k}, k\, u \right) \sim 
\left( \frac{k \, u}{\pi} \right)^d \cdot \left[ 1 
+ R_1\left(\frac{\vartheta}{ \sqrt{k} }, \frac{\mathbf{v}_j}{ \sqrt{k} } \right) \right].
\end{equation}

As a consequence, we have an asymptotic expansion 
\begin{eqnarray}
 \label{eqn:product asymptotic expansion}
\lefteqn{
e^{\imath\,k\cdot R_3\left(\frac{1}{\sqrt{k}} \, (\vartheta ,  \mathbf{v}_1, \mathbf{v}_2 )\right)} 
\cdot \left( e^{\imath\vartheta/\sqrt{k}}-e^{-\imath\,\vartheta/\sqrt{k}}\right)
\cdot s\left( \widetilde{\mu}_{g e^{-\imath \vartheta B/\sqrt{k}} g^{-1}}( x_{1k} ),
x_{2k}, k\, u \right)
}\nonumber\\
&\quad\quad\sim&\left( \frac{k \, u}{\pi} \right)^d \cdot \frac{2\,\imath}{\sqrt{ k }}\cdot \vartheta \cdot 
\left[1+\sum_{j\ge 1} k^{-j/2} \, P_j(x,u; \vartheta,\mathbf{v}_1,\mathbf{v}_2)\right],
\end{eqnarray}
where $P_j(x,u;\vartheta,\mathbf{v}_1,\mathbf{v}_2)$ is  
a polynomial of degree $\le 3j$ and parity $(-1)^j$.

Therefore, the amplitude in (\ref{eqn:Piknu and szego5}) is given by an asymptotic expansion in descending 
half-integer powers of $k$, and the asymptotic expansion for the integrand may be integrated terms by term. 
By a few computations, by (\ref{eqn:defn di Gamma}) 
one sees that the dominant term of the resulting expansion for 
(\ref{eqn:Piknu and szego5})
is the dominant term of the expansion for the following oscillatory integral:
\begin{eqnarray}
\label{eqn:Piknu and szego6-datogliere}
I_{\mathbf{v}_1, \mathbf{v}_2} (k) & := & 
\frac{\imath}{\sqrt{ k }}\cdot \left( \frac{k}{\pi} \right)^d \cdot\frac{k^{3/2}\,\nu }{(2\,\pi)^2}\, \int_0^{+\infty}\,\mathrm{d}u \,
\int_{-\infty}^{+\infty} \,\mathrm{d}\vartheta \,
\int_{-1}^{ 1 }\,\mathrm{ d }t \int_{-\pi}^{\pi}\, \mathrm{d} \delta\nonumber\\
&&\left[e^{\imath\, \sqrt{k}\, \Gamma(t ; u, \vartheta ) } \cdot e^{\imath \, u \, B_t 
+u \, \big[ \psi_2 \big( V_t, \mathbf{v}_2 \big) -\frac{ 1 }{ 2 } \, A_t^2     \big]}\right. \nonumber\\
&&\left. \cdot \rho  ( u ) \cdot \varrho_1\left(k^{-\epsilon_1}\,\vartheta\right) 
\cdot \vartheta \cdot u^d\right]\nonumber\\
&=& \frac{ \nu }{8\, \pi} \cdot \left( \frac{k}{\pi} \right)^{d+1} \,
\int_0^{+\infty}\,\mathrm{d}u \,
\int_{-\infty}^{+\infty} \,\mathrm{d}\vartheta \,
\int_{-1}^{ 1 }\,\mathrm{ d }t \int_{-\pi}^{\pi}\, \mathrm{d} \delta\nonumber\\
&&\left[e^{\imath\, \sqrt{k}\, \Gamma(t ; u, \vartheta ) } \cdot (2\imath \cdot \vartheta\cdot u)
\right. \nonumber\\
&&\left. \cdot e^{\imath \, u \, B_t 
+u \, \big[ \psi_2 \big( V_t, \mathbf{v}_2 \big) -\frac{ 1 }{ 2 } \, A_t^2     \big]}
\cdot \rho  ( u ) \cdot \varrho_1\left(k^{-\epsilon_1}\,\vartheta\right) 
 \cdot u^{d-1}\right].\nonumber
\end{eqnarray}
The latter may in turn be rewritten
\begin{eqnarray}
\label{eqn:Piknu and szego6}
I_{\mathbf{v}_1, \mathbf{v}_2}(k) 
&=& \frac{ \nu }{ 8\, \pi} \cdot \left( \frac{k}{\pi} \right)^{d+1} \,\frac{1}{\sqrt{k}\cdot \lambda (m_x)}\nonumber\\
&&\cdot \int_{-\pi}^{\pi}\, \mathrm{d} \delta \, 
\int_0^{+\infty}\,\mathrm{d}u \,
\int_{-\infty}^{+\infty} \,\mathrm{d}\vartheta \,
\int_{-1}^{ 1 }\,\mathrm{ d }t \,\left[ \partial_t \left( e^{\imath\, \sqrt{k}\, \Gamma(t;
u, \vartheta) } \right)
\right. \nonumber\\
&&\left. \cdot e^{\imath \, u \, B_t 
+u \, \big[ \psi_2 \big( V_t, \mathbf{v}_2 \big) -\frac{ 1 }{ 2 } \, A_t^2     \big]}
\cdot \varrho _1 ( u ) \cdot \varrho_2\left(k^{-\epsilon_1}\,\vartheta\right) 
 \cdot u^{d-1}\right].
\end{eqnarray}
Integrating by parts in $\mathrm{d}t$, we obtain
\begin{eqnarray}
\label{eqn:Piknu and szego7}
I_{\mathbf{v}_1, \mathbf{v}_2} (k) 
&=& \frac{ \nu }{ 8\, \pi} \cdot \left( \frac{k}{\pi} \right)^{d+1} \,\frac{1}{\sqrt{k}\cdot \lambda (m_x)}\\
&& \cdot \big[ J_{\mathbf{v}_1, \mathbf{v}_2}' (k) - J_{\mathbf{v}_1, \mathbf{v}_2}'' (k)
-  J_{\mathbf{v}_1, \mathbf{v}_2}''' (k)\big],\nonumber
\end{eqnarray}
where
\begin{eqnarray}
 \label{eqn:defn J'}
J_{\mathbf{v}_1, \mathbf{v}_2}' (k)
& := &
\int_{-\pi}^{\pi}\, \mathrm{d} \delta \, 
\int_0^{+\infty}\,\mathrm{d}u \,
\int_{-\infty}^{+\infty} \,\mathrm{d}\vartheta \,
\left[ e^{\imath\, \sqrt{k}\, \Gamma(1 ;
u, \vartheta) } 
\right. \\
&&\left. \cdot e^{\imath \, u \, B_1 
+u \, \big[ \psi_2 \big( V_1, \mathbf{v}_2 \big) -\frac{ 1 }{ 2 } \, A_1^2     \big]}
\cdot \rho ( u ) \cdot \varrho_1\left(k^{-\epsilon_1}\,\vartheta\right) 
 \cdot u^{d-1}\right], \nonumber
\end{eqnarray}

\begin{eqnarray}
 \label{eqn:defn J''}
J_{\mathbf{v}_1, \mathbf{v}_2}'' (k)
& := &
\int_{-\pi}^{\pi}\, \mathrm{d} \delta \, 
\int_0^{+\infty}\,\mathrm{d}u \,
\int_{-\infty}^{+\infty} \,\mathrm{d}\vartheta \,
\left[ e^{\imath\, \sqrt{k}\, \Gamma(-1 ;
u, \vartheta) } 
\right. \\
&&\left. \cdot e^{\imath \, u \, B_{-1 }
+u \, \big[ \psi_2 \big( V_{-1}, \mathbf{v}_2 \big) -\frac{ 1 }{ 2 } \, A_{-1}^2     \big]}
\cdot \rho  ( u ) \cdot \varrho_1\left(k^{-\epsilon_1}\,\vartheta\right) 
 \cdot u^{d-1}\right], \nonumber
\end{eqnarray}

\begin{eqnarray}
 \label{eqn:defn J'''}
J_{\mathbf{v}_1, \mathbf{v}_2}''' (k)
& := &
\int_{-1}^{ 1 }\,\mathrm{ d }t \,\int_{-\pi}^{\pi}\, \mathrm{d} \delta \, 
\int_0^{+\infty}\,\mathrm{d}u \,
\int_{-\infty}^{+\infty} \,\mathrm{d}\vartheta \,
\left[ e^{\imath\, \sqrt{k}\, \Gamma(t ;
u, \vartheta) } 
\right. \\
&&\left. \cdot \partial_t\left(  e^{\imath \, u \, B_{t }
+u \, \big[ \psi_2 \big( V_{t}, \mathbf{v}_2 \big) -\frac{ 1 }{ 2 } \, A_{t}^2     \big]}\right)
\cdot \rho  ( u ) \cdot \varrho_1\left(k^{-\epsilon_1}\,\vartheta\right) 
 \cdot u^{d-1}\right]. \nonumber
\end{eqnarray}

Let us estimate the three terms (\ref{eqn:defn J'}), (\ref{eqn:defn J''})
and (\ref{eqn:defn J'''}) separately.

\begin{lem}
 \label{lem:J' expansion}
As $k\rightarrow +\infty$, there is an asymptotic expansion of the form
\begin{eqnarray*}
J_{\mathbf{v}_1, \mathbf{v}_2}' (k)
& \sim & 
\frac{4\, \pi^2}{\sqrt{k}}\cdot \frac{1}{2\,\lambda (m_x)}\cdot e^{u_0(\nu,m_x) \cdot \psi_2 (\mathbf{v}_1, \mathbf{v}_2)}
\cdot \left(\frac{\nu}{2\,\lambda (m_x)}\right)^{d-1}\\
&&\cdot\left[1+\sum_{l=1}^{+\infty} k^{-l/2}\,a_l(m_x; \mathbf{v}_1,\mathbf{v}_2)\right],
\end{eqnarray*}
where $a_l(m_x; \cdot,\cdot)$ is a polynomial of degree $\le 3l$ and parity $(-1)^l$, whose coefficients are
$\mathcal{C}^\infty$ functions on $M$.
\end{lem}

\begin{proof}
 [Proof of Lemma \ref{lem:J' expansion}]
Let us view $J_{\mathbf{v}_1, \mathbf{v}_2}' (k)$ as an oscillatory integral in the
parameter $\sqrt{k}$, with real phase 
\begin{eqnarray}
 \label{eqn:defn di Gamma1}
\Gamma(1; u, \vartheta) 
&=& \vartheta\cdot \big[ 2 \, \lambda (m_x) \cdot u 
- \nu\big],
\end{eqnarray}
and amplitude
\begin{eqnarray} 
\label{eqn:amplitudet1}
 e^{\imath \, u \, B_1 
+u \, \big[ \psi_2 \big( V_1, \mathbf{v}_2 \big) -\frac{ 1 }{ 2 } \, A_1^2     \big]}
\cdot \varrho _1 ( u ) \cdot \varrho_2\left(k^{-\epsilon_1}\,\vartheta\right) 
 \cdot u^{d-1}.
\end{eqnarray}
Explicitly, the exponent is
\begin{eqnarray}
 \label{eqn:exponent t1}
\mathcal{E}_1(u,\vartheta, \mathbf{v}_1, \mathbf{v}_2) & := & \imath \, u \, B_1 
+u \, \big[ \psi_2 \big( V_1, \mathbf{v}_2 \big) -\frac{ 1 }{ 2 } \, A_1^2     \big]
\\
&=& u \, \left[-\imath\, \omega_{ m_x } (\mathbf{v}_1,  \mathbf{v}_2 ) 
+ \imath \, \vartheta\, \omega_{ m_x } \Big( \mathrm{Ad}_{h_{m_x}}( \beta)_M (m_x), \mathbf{v}_1 +\mathbf{v}_2\Big)\right. \nonumber \\
&& \left.
-\frac{1}{2} \, \Big\| (\mathbf{v}_1 - \mathbf{v}_2) -\vartheta \, \mathrm{Ad}_{h_{m_x}}(\beta)_X(x) \Big\|^2\right].\nonumber
\end{eqnarray}
Under the hypothesis of the Theorem, therefore, $\Re (\mathcal{E}_1) \le - C'\, \vartheta^2 +C''$ for some 
constants $C',\, C''>0$. 

Furthermore, the phase has a unique critical point, given by
$(u_0(\nu,m_x), 0)$ (Definition \ref{defn:coset}),
and Hessian matrix
$$
\mathrm{Hess}\big( \Gamma(1 ;
\cdot, \cdot)\big)= 
\begin{pmatrix}
 0& 2\,\lambda (m_x) \\
2\,\lambda (m_x) & 0
\end{pmatrix}.
$$
Hence the Hessian determinant is $-4\, \lambda (m_x)^2$ and its signature is zero. 
Thus the critical point is non-degenerate, and the critical value is
$
\Gamma(1 ;
u_0, 0) =0$.
At the critical point, the exponent in the amplitude is 
$$
\mathcal{E}_1\big(u_0(\nu,m_x) ,0 , \mathbf{v}_1, \mathbf{v}_2\big)
=u_0(\nu,m_x) \cdot \psi_2 (\mathbf{v}_1, \mathbf{v}_2).
$$
When applying the Stationary Phase Lemma, at the $l$-th step we need to let
the differential operator 
$k^{-l/2}\,R_\Gamma ^l$ act on the amplitude (\ref{eqn:amplitudet1}), where
$$
R_\Gamma:= \frac{\imath}{4\cdot \lambda(m_x)}\cdot \frac{\partial^2}{\partial u\,\partial\vartheta},
$$ 
and evaluate the result at the critical point. One sees inductively that 
\begin{eqnarray*}
k^{-l/2}\,R_\Gamma ^l\left( e^{\imath \, \mathcal{E}_1(u,\vartheta, \mathbf{v}_1, \mathbf{v}_2)}\right)
= H_{l}\,e^{\imath \,\mathcal{E}_1(u,\vartheta, \mathbf{v}_1, \mathbf{v}_2)},
\end{eqnarray*}
where $H_l$ is a polynomial of degree $\le 3l$ in $(\vartheta,\mathbf{v}_1,\mathbf{v}_2)$ and parity $(-1)^l$.

The proof of Lemma \ref{lem:J' expansion} is complete.

\end{proof}

\begin{lem}
 \label{lem:J'' negligible}
As $k\rightarrow +\infty$, we have $J_{\mathbf{v}_1, \mathbf{v}_2}'' (k) = O\left(k^{-\infty}\right)$.
\end{lem}

\begin{proof}
[Proof of Lemma \ref{lem:J'' negligible}]
Let us view $J_{\mathbf{v}_1, \mathbf{v}_2}'' (k)$ as an oscillatory integral in $\sqrt{k}$, with 
phase $\Gamma_{\theta_1,\theta_2}(-1 ;
u, \vartheta)$. By (\ref{eqn:defn di Gamma}),
$$
\partial_{\vartheta} \Gamma(-1 ;
u, \vartheta)=-2\, u\cdot  \lambda (m_x) -\nu\le -\nu.
$$
The claim follows by 
iterated integration by parts in $\vartheta$.
\end{proof}

\begin{lem}
 \label{lem:J''' lower order}
$\mathcal{J}_{\mathbf{v}_1, \mathbf{v}_2}''' (k;t,\delta)_2=O\left(k^{-1}\right)$ as $k\rightarrow +\infty$; furthermore,
$\mathcal{J}_{\mathbf{v}_1, \mathbf{v}_2}''' (k;t,\delta)_2$ admits an asymptotic expansion of the same kind as 
$\mathcal{J}_{\mathbf{v}_1, \mathbf{v}_2}' (k;t,\delta)_2$ (of lower leading order).
\end{lem}

\begin{proof}
[Proof of Lemma \ref{lem:J''' lower order}]
Let us choose $\epsilon'\in \big(0, \nu/ \big(4\,D\cdot \lambda(m_x)\big)\big)$, and consider the open cover 
$\mathcal{U}:=\{[-1,2\epsilon'),(\epsilon',1]\}$.
Let $\gamma_1(t)+\gamma_2(t)=1$ be a smooth partition of unity on $[-1,1]$
subordinate to $\mathcal{U}$.
Thus 
$$
J_{\mathbf{v}_1, \mathbf{v}_2}''' (k) = J_{\mathbf{v}_1, \mathbf{v}_2}''' (k)_1 + J_{\mathbf{v}_1, \mathbf{v}_2}''' (k)_2,
$$
where $J_{\upsilon_1, \upsilon_2}''' (k)_j$ is defined as in
(\ref{eqn:defn J'''}), with the extra factor
$\gamma_j(t)$. Explicitly, let us set 
\begin{eqnarray}
 \label{eqn:exponent tt}
\mathcal{E}_t(u,\vartheta, \mathbf{v}_1, \mathbf{v}_2) & := & \imath \, u \, B_t 
+u \, \big[ \psi_2 \big( V_t, \mathbf{v}_2 \big) -\frac{ 1 }{ 2 } \, A_t^2     \big].
\end{eqnarray}
Then 
$$
J_{\mathbf{v}_1, \mathbf{v}_2}''' (k)_j
=\int_{-\pi}^{\pi}\, \mathrm{d} \delta \,\int_{-1}^{ 1 }\,\mathrm{ d }t \,
\left[\mathcal{J}_{\mathbf{v}_1, \mathbf{v}_2}''' (k;t,\delta)_j\right],
$$
where
\begin{eqnarray}
\mathcal{J}_{\mathbf{v}_1, \mathbf{v}_2}''' (k;t,\delta)_j&=&\int_0^{+\infty}\,\mathrm{d}u \,
\int_{-\infty}^{+\infty} \,\mathrm{d}\vartheta \,
\left[ e^{\imath\, \sqrt{k}\, \Gamma(t ;
u, \vartheta) } \cdot \gamma_j (t)
\right. \\
&&\left. \cdot \partial_t\Big(\mathcal{E}_t(u,\vartheta, \mathbf{v}_1, \mathbf{v}_2)\Big)\cdot 
 e^{\mathcal{E}_t(u,\vartheta, \mathbf{v}_1, \mathbf{v}_2)}
\cdot \rho ( u ) \cdot \varrho_1\left(k^{-\epsilon_1}\,\vartheta\right) 
 \cdot u^{d-1}\right]. \nonumber
\end{eqnarray}
Let us view $\mathcal{J}_{\mathbf{v}_1, \mathbf{v}_2}''' (k;t,\delta)_j$ as an oscillatory integral with phase 
$\Gamma_t(u,\vartheta):=\Gamma(t ;
u, \vartheta) $.

On the support of $\rho(u)\cdot \gamma_1(t)$, we have $u\le D$ and $t\le \epsilon'$; therefore,
$$
\partial_{\vartheta}\Gamma_t(
u, \vartheta)= 2 u \cdot t \cdot \lambda (m_x) 
- \nu \le 2\, D\cdot \epsilon'\cdot \lambda (m_x) -\nu \le -\frac{\nu}{2}.
$$
Therefore, integration by parts in $\vartheta$ implies that $\mathcal{J}_{\mathbf{v}_1, \mathbf{v}_2}''' (k;t,\delta)_j
=O\left(k^{-\infty}\right)$, uniformly for $t$ in the support of $\gamma_1$. Hence 
$J_{\mathbf{v}_1, \mathbf{v}_2}''' (k)_1=O\left(k^{-\infty}\right)$.

On the support of $\gamma_2$, on the other hand,
$\Gamma(t; \cdot, \cdot)$ has the non-degenerate critical point
\begin{equation*}
 \big( u(t), 0 \big) = \left( \frac{ \nu }{2\, t\, \lambda (m_x)},
0\right),
\end{equation*}
with Hessian matrix 
$$
\mathrm{Hess}\big( \Gamma_t\big)= 
\begin{pmatrix}
 0& 2\,t\cdot \lambda (m_x) \\
2\,t\cdot \lambda (m_x) & 0
\end{pmatrix}.
$$
Therefore, we can apply the Stationary Phase Lemma as in the proof of Lemma \ref{lem:J' expansion},
viewing $t$ as a parameter. The asymptotic expansion will be non trivial only for 
$t\in \big[\nu/(2\,\lambda (m_x)\,D),1\big]$, for otherwise $\rho $
vanishes identically in a neighborhood of $u(t)$.

Given (\ref{eqn:defn of Vk}), (\ref{eqn:defn di A}), and (\ref{eqn:defn di B})
(with $g\,T$ replaced by $h_{m_x}\,g(t,\delta)\,T$), the integrand in (\ref{eqn:defn J'''}) is divisible by
$\vartheta$; hence it vanishes at the critical point. Furthermore, the integrand is of class $L^1$ as a function
of the parameter $t$, since the exponent getting differentiated is a smooth function of $t$ and $\sqrt{1-t^2}$.

Hence $\mathcal{J}_{\mathbf{v}_1, \mathbf{v}_2}''' (k;t,\delta)_2$ admits an asymptotic expansion in descending 
half-integer powers of $k$, with leading power $k^{-1}$, and coefficients of class $L^1$
as functions of $t$. The general term of the expansion will be a scalar multiple of 
\begin{equation}
 \label{eqn:general term}
k^{-(l+1)/2} \cdot R_{\Gamma_t}^l\left(  \partial_t\Big(\mathcal{E}_t(u,\vartheta, \mathbf{v}_1, \mathbf{v}_2)\Big)\cdot 
 e^{\mathcal{E}_t(u,\vartheta, \mathbf{v}_1, \mathbf{v}_2)}\right).
\end{equation}

Given integers $a,b\ge 0$, let us denote by $H_{a,b}(\vartheta;\mathbf{v}_1,\mathbf{v}_2)$ a generic polynomial 
in $(\vartheta,\mathbf{v}_1,\mathbf{v}_2)$,
which is separately homogeneous of degree $a$ in $\vartheta$, and of degree $b$ in $(\mathbf{v}_1,\mathbf{v}_2)$,
and by $H_{a}(\vartheta,\mathbf{v}_1,\mathbf{v}_2)$ a generic polynomial in $ (\vartheta;\mathbf{v}_1,\mathbf{v}_2)$
homogeneous of degree $a$ (but perhaps not polyhomogeneous); both $H_{a,b}$ and $H_a$ are allowed to vary from line to line,
and their coefficients depend smoothly on $u$.
Thus
$\mathcal{E}_t=u\cdot H_{2}=u\cdot (H_{2,0}+H_{1,1}+H_{0,2})$, $\partial_t\mathcal{E}_t=u(H_{2,0}+H_{1,1})$ (here the polynomials do not depend on $u$).
Hence we can split (\ref{eqn:general term}) as 
\begin{eqnarray}
 \label{eqn:general term1}
\lefteqn{k^{-(l+1)/2} \cdot \left[R_{\Gamma_t}^l\left(  \rho(u)\cdot H_{2,0}\cdot 
 e^{u\cdot (H_{2,0}+H_{1,1}+H_{0,2})}\right)\right.}\nonumber\\
&&\left.+ R_{\Gamma_t}^l\left( \rho(u)\cdot H_{1,1}\cdot 
 e^{u\cdot (H_{2,0}+H_{1,1}+H_{0,2})}\right)\right].
\end{eqnarray}

The proof of Lemma \ref{lem:J''' lower order}
is completed by the following two claims, which can be proved inductively from the cases $l=0,1$:

\begin{claim}
\label{claim:first case}
 For $l=0,1,2,\ldots$,
$$
R_{\Gamma_t}^l\left( \rho(u)\cdot H_{1,1}\cdot 
 e^{\mathcal{E}_t}\right)$$
is a sum of terms of the form
$$
\Big[H_{0,1}\cdot H_{p_l}+H_{1,1}\cdot H_{q_l}\Big]\cdot e^{\mathcal{E}_t},
$$
where $p_l+1\le 3\,l$, $(-1)^{p_l+1}=(-1)^{l}$, and $q_l\le 3l$, $(-1)^{q_l}=(-1)^l$.
\end{claim}

\begin{claim}
 \label{claim:second case}
For $l=0,1,2,\ldots$,
$$
R_{\Gamma_t}^l\left( \rho(u)\cdot H_{2,0}\cdot 
 e^{\mathcal{E}_t}\right)$$
is a sum of terms of the form
$
H_{a,b}\cdot e^{\mathcal{E}_t}
$,
where $b\le 3\,l$ and $(-1)^{a+b}=(-1)^{l}$.
\end{claim}

\end{proof}

Since at the critical point 
$\vartheta=0$, the summands with a factor of the form $H_{a,b}$ with $a\ge 1$ all vanish at the critical point.
It follows that the asymptotic expansion for (\ref{eqn:Piknu and szego6}) is as in the statement of Proposition
\ref{prop:rescaled case 0}.
The contributions to the asymptotic expansion for (\ref{eqn:Piknu and szego5}) coming from the lower order terms
in (\ref{eqn:product asymptotic expansion}) can be dealt with by similar arguments. This completes the proof of 
Proposition \ref{prop:rescaled case 0}.
\end{proof}

\subsubsection{$\Pi_{k \boldsymbol{\nu} }
\big (x_{1k} , x_{2k} \big)_-$}

Let us now consider the asymptotics of the second summand in
(\ref{eqn:Pik decomposed 0pi}). 
We shall prove the following analogue of Proposition \ref{prop:rescaled case 0}.

\begin{prop}
\label{prop:rescaled case 1}
Under the assumptions of Theorem \ref{thm:rescaled asymptotics},
suppose in addition that $-I_2\in G_x$. 
Let us set 
$\mathbf{v}_1^{-I_2} := 
\mathrm{d}_{m_x}\mu_{-I_2}(\mathbf{v}_1)$.
Then
as $k\rightarrow +\infty$ we have an asymptotic expansion
\begin{eqnarray*}
\Pi_{k\boldsymbol{\nu}}\left(x_{1k} , x_{2k}\right)_-
& \sim &  \frac{e^{\imath\,\pi (1-k\cdot  \nu)}}{2\,\lambda (m_x)} \cdot \left(\frac{\nu \,k }{2\,\pi\,\lambda (m_x)}\right)^d \cdot 
e^{u_0(\nu,m_x) \cdot \psi_2 (\mathbf{v}_1^{-I_2}, \mathbf{v}_2) }\\
&& \cdot 
\left[1+ \sum_{j\ge 1}^{+\infty} k^{-j/2} A^-_j(x;\mathbf{v}_1, \mathbf{v}_2)\right],
\end{eqnarray*}
where
$A^-_j(x;\cdot, \cdot)$ is a polynomial of degree $\le 3\,j$ and parity $(-1)^j$.
\end{prop}

\begin{proof}
 [Proof of Proposition \ref{prop:rescaled case 1}]
The proof is a slight modification of the one for Proposition \ref{prop:rescaled case 0}.

$\Pi_{k \boldsymbol{\nu} }
\big (x_{1k} , x_{2k} \big)_-$ is given by (\ref{eqn:Piknu and szego1}), 
with $\varrho$
replaced by $\varrho_- $; therefore,
integration in $\mathrm{d}\vartheta$ is now restricted to $(\pi-2\,\delta,\pi+2\,\delta)$. 
With the change of variable
$\vartheta\mapsto \vartheta+\pi$, $g e^{-\imath \vartheta B} g^{-1}$ gets replaced by
$-g e^{-\imath \vartheta B} g^{-1}$ and $\vartheta\in (-\delta,\delta)$.


Lemma \ref{lem:shrinking support vartheta} still applies, so that we can again rescale in $\vartheta$.
By (\ref{eqn:invariant HLC}), we have
\begin{eqnarray}
 \label{eqn:differential of -I_2}
x^{-I_2}_{1k}&:=&\widetilde{\mu}_{-I_2}( x_{1k} )= 
x+\frac{1}{\sqrt{k}}\,\mathbf{v}_1^{-I_2}.
\end{eqnarray}
Therefore, in place of (\ref{eqn:Piknu and szego2}), we obtain the following:
\begin{eqnarray}
\label{eqn:Piknu and szego2pi}
\lefteqn{
\Pi_{k \boldsymbol{\nu} }
\big (x_{1k} , x_{2k} \big)_-
} \\
&\sim & \frac{k^{3/2}\,\nu }{2\pi}\, \int_0^{+\infty}\,\mathrm{d}u \,
\int_{-\infty}^{+\infty} \,\mathrm{d}\vartheta \,
\int_{G/T} \, \mathrm{d}V_{G/T} (gT)
\left[e^{\imath\, k\, \Gamma_k} \cdot \varrho_2\left(k^{-\epsilon_1}\,\vartheta\right) \right.
\nonumber\\
&&\left.\cdot \rho ( u ) \cdot
\left( e^{\imath\left(\pi+\vartheta/\sqrt{k}\right)}-e^{-\imath\,\left(\pi+\vartheta/\sqrt{k}\right)}\right)\cdot 
s\left( \widetilde{\mu}_{g e^{-\imath \vartheta B/\sqrt{k}} g^{-1}}( x^{-I_2}_{1k} ),
x_{2k}, k\, u \right)
\right] \nonumber,
\end{eqnarray}
where
\begin{eqnarray}
 \label{eqn:defn di Psikpi}
\Gamma_k(u, \mathbf{v}_1 , \mathbf{v}_2 ,\vartheta, g\,T) & := &
u\, \psi \left( \widetilde{\mu}_{g e^{-\imath \vartheta B/\sqrt{k}} g^{-1}}
\circ ( x^{-I_2}_{1k} ),
x_{2k} \right) -\frac{\vartheta}{\sqrt{k}}\cdot \nu-\pi\cdot \nu\nonumber\\
&=&\Psi_k (u, \mathbf{v}_1^{-I_2} , \mathbf{v}_2 ,\vartheta, g\,T)-\pi\cdot \nu.
\end{eqnarray}

Let us write 
$\Psi_k':=\Psi_k (u, \upsilon_1' , \upsilon_2 ,\vartheta, g\,T)$.
Thus we may rewrite (\ref{eqn:Piknu and szego2}) in the following manner:
\begin{eqnarray}
\label{eqn:Piknu and szego2pi1}
\lefteqn{
\Pi_{k \boldsymbol{\nu} }
\big (x_{1k} , x_{2k} \big)_-
} \\
&\sim & e^{\imath\,\pi (1-k\cdot  \nu)}\cdot \frac{k^{3/2}\,\nu }{2\pi}\, \int_0^{+\infty}\,\mathrm{d}u \,
\int_{-\infty}^{+\infty} \,\mathrm{d}\vartheta \,
\int_{G/T} \, \mathrm{d}V_{G/T} (gT)
\left[e^{\imath\, k\, \Psi_k'} \cdot \varrho_2\left(k^{-\epsilon_1}\,\vartheta\right)\right.
\nonumber\\
&&\left.\cdot \varrho _1 ( u ) \cdot
\left( e^{\imath\,\vartheta/\sqrt{k}}-e^{-\imath\,\vartheta/\sqrt{k}}\right)\cdot 
s\left( \widetilde{\mu}_{g e^{-\imath \vartheta B/\sqrt{k}} g^{-1}}( x'_{1k} ),
x_{2k}, k\, u \right)
\right] \nonumber\\
&\sim& e^{\imath\,\pi (1-k\cdot \nu)}\cdot  \Pi_{k \boldsymbol{\nu} }
\left (x+\frac{1}{\sqrt{k}}\,\mathbf{v}_1^{-I_2}, x+\frac{1}{\sqrt{k}}\,\mathbf{v}_2 \right)_+.
\end{eqnarray}

The statement of Proposition \ref{prop:rescaled case 1} follows from (\ref{eqn:Piknu and szego2pi1})
and Proposition \ref{prop:rescaled case 0}.

\end{proof}

The proof of Theorem \ref{thm:rescaled asymptotics} is complete.

\end{proof}


\subsection{Theorem \ref{thm:G_xnotinZx}}

\begin{proof}
 [Proof of Theorem \ref{thm:G_xnotinZx}]

Let $\varrho:G\rightarrow [0,+\infty)$ be a smooth bump function supported in a small neighborhood of $I_2$,
and identically equal to $1$ on a smaller neighborhood. With $g_j$ as in 
(\ref{eqn:defn di t_j}), $j=1,\ldots,N_x$, let us set 
$\varrho_j(g):=\varrho\left(g\,g_j^{-1}\right)$.
Then
\begin{eqnarray}
 \label{eqn:PiknintegraleG_x}
\Pi_{k\,\boldsymbol{\nu}}(x,x)
& = & k\,\nu \cdot\int_G\,\mathrm{d}V_G(g) \,\left[\overline{\chi_{k\,\boldsymbol{\nu}}(g)}\,
\Pi \left(\widetilde{\mu}_{g^{-1}} (x) ,x\right)\right] \nonumber\\
& \sim&
\sum_{j=1}^{N_x} k\,\nu \cdot \int_G\,\mathrm{d}V_G(g) \,\left[\varrho_j(g)\cdot
\overline{\chi_{k\,\boldsymbol{\nu}}(g)}\,
\Pi \left(\widetilde{\mu}_{g^{-1}} (x) ,x\right)\right] .
\end{eqnarray}
Let us write $\Pi_{k\,\boldsymbol{\nu}}(x,x)_j$ for the $j$-th summand in (\ref{eqn:PiknintegraleG_x}).

With $Z_x$ as in Definition \ref{defn:center}, we can rewrite (\ref{eqn:PiknintegraleG_x}) as follows:
\begin{eqnarray}
 \label{eqn:PiknintegraleG_x two types}
\Pi_{k\,\boldsymbol{\nu}}(x,x)
&\sim&\Pi_{k\,\boldsymbol{\nu}}(x,x)_{Z_x}
+\Pi_{k\,\boldsymbol{\nu}}(x,x)_{G_x\setminus Z_x},
\end{eqnarray}
where
\begin{equation}
 \label{eqn:Z_x+nonZ_x}
\Pi_{k\,\boldsymbol{\nu}}(x,x)_{Z_x}:=\sum_{g_j\in Z_x}\Pi_{k\,\boldsymbol{\nu}}(x,x)_j,
\quad \Pi_{k\,\boldsymbol{\nu}}(x,x)_{G_x\setminus Z_x}:=\sum_{g_j\not\in Z_x}\Pi_{k\,\boldsymbol{\nu}}(x,x)_j.
\end{equation}

The asymptotic expansion (\ref{eqn:Z_x term}) for the first summand in (\ref{eqn:PiknintegraleG_x two types}) 
follows from Theorem \ref{thm:rescaled asymptotics},
with $\mathbf{v}_1=\mathbf{v}_2=0$. Hence, Theorem \ref{thm:G_xnotinZx} will be proved by establishing the following.

\begin{prop}
 \label{prop:g_jnotinZx}
Assume, as in (\ref{eqn:G_xminusZ_x}),
that $G_x\setminus Z_x=\left\{g_j,g_{j+a_x}:=g_j^{-1}\right\}_{j=1}^{a_x}$, and let
$B(x;j)$ be as in Definition \ref{defn:definition of Bxj}. Then
as $k\rightarrow +\infty$ there is an asymptotic expansion
\begin{eqnarray*}
\Pi_{k\,\boldsymbol{\nu}}(x,x)_{G_x\setminus Z_x}
& \sim &4\,\pi\left(\frac{\nu\,k}{2\,\pi\cdot \lambda(m_x)}\right)^d\\
&&\cdot 
\left[\sum_{j=1}^{a_x}\Re\left(\frac{\imath\,\sin(\vartheta_j)\cdot e^{-\imath k \nu\cdot \vartheta_j}}{\sqrt{\det \big(B(x;j)\big)}}\right)
+\sum_{l\ge 1} k^{-l/2}\,P_{jl}(\nu;m_x)\right],
\end{eqnarray*}
where $P_{jl}(\nu;\cdot)$ is $\mathcal{C}^\infty$ on the loci (in $M$) defined by the cardinality of $G_x$, $x\in p^{-1}(m)$.
\end{prop}

\begin{proof}
 [Proof of Proposition \ref{prop:g_jnotinZx}]
Let $\rho:G/T \times T\rightarrow G$, $(g\,T,t)\mapsto g\,t\,g^{-1}$; each $g_j\in G_x\setminus Z_x$, being a regular element of
$G$, is a regular value of $\rho$. If $t_j$ is as in (\ref{eqn:defn di t_j}),
then
\begin{equation}
 \label{eqn:inverse imagepj}
\rho^{-1} (g_j) =\left\{(h_{m_x}\,T, t_j),\, (k_{m_x}\,T,t_j^{-1})\right\},\quad 
k_{m_x}:= h_{m_x} \,
\begin{pmatrix}
 0&-1\\
1&0
\end{pmatrix}.
\end{equation}
Let $E:\xi\in \mathfrak{g}\mapsto e^\xi\in G$ be the exponential map.  
We shall write the general $t\in T$ in exponential form as
$$
t=
\begin{pmatrix}
 e^{\imath\vartheta} & 0 \\
0 & e^{-\imath\vartheta} 
\end{pmatrix} =
E(\vartheta\,\beta),
$$
where $\beta$ is as in Remark \ref{rem:positive eigenvalue}.

By the Weyl integration and character formulae, we have
\begin{eqnarray}
 \label{eqn:jthsummand}
\Pi_{k\,\boldsymbol{\nu}}(x,x)_j  
&=& \frac{k\,\nu}{2\pi} \cdot \int_{G/T}\,\mathrm{d}V_{G/T}(g\,T)\,\int_{-\pi}^{\pi}\mathrm{d}\vartheta\\
&&\left[\varrho_j\left(g\,e^{\vartheta \beta}\,g^{-1}\right)\cdot
e^{-\imath k \nu\cdot \vartheta}\,
\Pi \left(\widetilde{\mu}_{g e^{- \vartheta \beta}g^{-1}} (x) ,x\right)\,
\left(e^{\imath\,\vartheta}-e^{-\imath\,\vartheta}\right)\right].\nonumber
\end{eqnarray}

Since $\varrho_j\circ \rho:\left(g\,T, e^{\imath\,\vartheta}\right)\mapsto \varrho_j\left(g\,e^{\vartheta \beta}\,g^{-1}\right)$
is supported in a small open neighborhood of the pair (\ref{eqn:inverse imagepj}),
we can split (\ref{eqn:jthsummand}) as 
\begin{equation}
 \label{eqn:jthsummand12}
\Pi_{k\,\boldsymbol{\nu}}(x,x)_j = \Pi_{k\,\boldsymbol{\nu}}(x,x)_{j1} + \Pi_{k\,\boldsymbol{\nu}}(x,x)_{j2},
\end{equation}
where in $\Pi_{k\,\boldsymbol{\nu}}(x,x)_{j1}$ (respectively, $\Pi_{k\,\boldsymbol{\nu}}(x,x)_{j2}$)
integration is over a small neighborhood of 
$(h_{m_x}\,T, t_j)$ (respectively, $(k_{m_x}\,T,t_j^{-1})$).

Let us first consider each $\Pi_{k\,\boldsymbol{\nu}}(x,x)_{j1}$.
To this end, we shall introduce local coordinates on $G/T$ and on $T$.

First, for some suitably small $\delta>0$ and $z\in D(0,\delta)\subset \mathbb{C}$, we set
\begin{equation}
 \label{eqn:defn of h}
h(z):= h_{m_x}\,E\big(A(z)\big),
\end{equation}
where $A(z)$ is as in Definition \ref{defn:definition of Bxj};
then 
\begin{equation}
 \label{eqn:local coordinates G/T}
z\in D(0,\delta)\mapsto h(z)\,T\in G/T
\end{equation}
is a system of local coordinates on 
$G/T$ centered at $h_{m_x}\,T$. The Haar measure on $G/T$, expressed in the $z$ coordinates,
is $\mathcal{V}_{G/T}(z)\,\mathrm{d}V_{\mathbb{C}}(z)$, for an appropriate smooth function 
on $\mathcal{V}_{G/T}$. 
The proof of the following Lemma will be omitted.

\begin{lem}
\label{lem:DG/T}
Let us set
$D_{G/T}=\mathcal{V}_{G/T}(0)$. Then 
$D_{G/T}=2\pi/V_3$, where $V_3$ is the surface area of $S^3$.
\end{lem}

Next, as a system of local coordinates on $T$ centered at $t_j$ we shall adopt
$
\theta\in (-\delta,\delta)\mapsto t_j\,E(\theta\,\beta)\in T
$. 
Furthermore, since $\left(\widetilde{\mu}_{g e^{- \vartheta \beta}g^{-1}} (x) ,x\right)$ is in a small neighborhood
of the diagonal in $X\times X$, we may replace $\Pi$ by its representation as an FIO. 
After performing the rescaling $u\mapsto ku$, and recalling Proposition \ref{prop:compact support u},
we obtain
\begin{eqnarray}
 \label{eqn:jthsummand1}
\lefteqn{ \Pi_{k\,\boldsymbol{\nu}}(x,x)_{j1}   }\\
&\sim& \frac{k^2\,\nu}{2\pi} \cdot 
e^{-\imath k \nu\cdot \vartheta_j}\cdot \int_{D(0,\delta)}\,\mathrm{d}V_{\mathbb{C}}(z)\,\int_{-\delta}^{\delta}\mathrm{d}\theta
\,\int_0^{+\infty}\mathrm{d}u\nonumber\\
&&\left[
e^{\imath\,k\,\left[u \,\psi\left(\widetilde{\mu}_{h(z) E(- \theta \beta) t_j^{-1} h(z)^{-1}} (x) ,x\right)-\nu\cdot \theta\right]}
\cdot \mathcal{V}_{G/T}(z)
\right.\nonumber\\
&&\left. \cdot \rho(u)\cdot 
s_{j1} \left(\widetilde{\mu}_{h(z) E(- \theta \beta) t_j^{-1} h(z)^{-1}} (x) ,x,k\,u\right)\cdot
\left(e^{\imath\,(\vartheta_j+\theta)}-e^{-\imath\,(\vartheta_j+\theta)}\right)\right].\nonumber
\end{eqnarray}
Here, $\mathrm{d}V_{\mathbb{C}}(z)$ is the Lebesgue measure on $\mathbb{C}\cong \mathbb{R}^2$,
and $s_{j1}$ denotes the usual amplitude of the representation of $\Pi$ as an FIO, with the above cut-offs incorporated.

In order to proceed, we need to express the phase more explicitly. 
We have
\begin{eqnarray}
 \label{eqn:exponential orm j1}
\lefteqn{h(z) E(- \theta \beta) t_j^{-1} h(z)^{-1}} \\
&=&C_{h_{m_x}}\Big(E\big(A(z)\big)\, E(- \theta \beta) E\big(-\mathrm{Ad}_{t_j^{-1}}\big(  A(z)\big)\big)\Big)\,g_j^{-1} \nonumber\\
&=& E \Big (-\mathrm{Ad}_{h_{m_x}}\big(\gamma_j (z,\theta)\big)\Big)\,g_j^{-1},\nonumber
\end{eqnarray}
where (by use of the Baker-Campbell-Hausdorff formula)
$$
\gamma_j (z,\theta)= \gamma_{j1} (z,\theta) + \gamma_{j2}(z,\theta)+ R_3(z,\theta),
$$
with 
\begin{eqnarray}
 \label{eqn:defn di gamma1e2}
\gamma_{j1} (z,\theta) & := & 
\theta\,\beta+ \big(\mathrm{Ad}_{t_j^{-1}} - \mathrm{id}_{\mathfrak{g}}\big)\big( A(z)\big),\\
\gamma_{j2} (z,\theta) & := & - \frac{1}{ 2 } \,\big[\theta \, \beta, A(z)+ \mathrm{Ad}_{t_j^{-1}}\big(  A(z)\big) \big]\nonumber\\
&&+\frac{1}{2} \,\big[ A(z) , \mathrm{Ad}_{t_j^{-1}}\big(A(z)\big)\big],\nonumber
\end{eqnarray}
while $R_j$ denotes a generic $\mathcal{C}^\infty$ function vanishing to $j$-th order at the origin. 
Note that $\gamma_j$ is homogeneous of degree $j$ in $\big(\Re(z),\Im(z),\theta)$,

By Corollary 2.2 of
\cite{pao-IJM}, we obtain in HLC's 
\begin{eqnarray}
\label{eqn:muxzthetaHLCsj1}
\widetilde{\mu}_{h(z) E(- \theta \beta) t_j^{-1} h(z)^{-1}} (x) 
& = & \widetilde{\mu}_{E  \big(-\mathrm{Ad}_{h_{m_x}}\big(\gamma_j (z,\theta)\big)\big)} (x)    \nonumber\\
&=& x+\Big (\Theta (z,\theta), V(\theta,z)\Big),
\end{eqnarray}
where
\begin{eqnarray}
 \label{eqn:defn di Theta V:=}
\Theta (z,\theta)& := & \Big\langle \Phi_G (m_x), \mathrm{Ad}_{h_{m_x}}\big(\gamma_j (z,\theta)\big)\Big\rangle + R_3(z,\theta)\\
V(\theta,z) & := & -\mathrm{Ad}_{h_{m_x}}\big(\gamma_j (z,\theta)\big)_M (m) +R_2(z,\theta).\nonumber
\end{eqnarray}

By the discussion in \S 3 of \cite{sz}, we conclude that 
\begin{eqnarray}
 \label{eqn:psi expanded j1}
\lefteqn{ 
u \,\psi\left(\widetilde{\mu}_{h(z) E(- \vartheta \beta) t_j^{-1} h(z)^{-1}} (x) ,x\right)-\nu\cdot \theta
} \nonumber\\
& = &u \,\psi\left(x+\Big (\Theta (z,\theta), V(\theta,z)\Big) ,x\right)-\nu\cdot \theta \nonumber\\
&=& \imath\, u \cdot \left[ 1- e^{\imath\,\Theta (z,\theta)}\right] 
+\frac{\imath\,u }{2}\cdot \big\|V(\theta,z)\big\|^2+u\,R_3(z,\theta)\nonumber	\\
&=& u\,\Theta(z,\theta) + \frac{\imath \, u}{2} \cdot \left[\Theta(z,\theta)^2+\big\|V(\theta,z)\big\|^2\right]
+ u\,R_3(z,\theta).
\end{eqnarray}

Let us choose $C>0$, $\epsilon \in (0,1/6)$. Since $p$ is a local diffeomorphism at $(h_{m_x}\,T, t_j)$
and $\widetilde{\mu}$ is locally free at $x$, the contribution to the asymptotics of (\ref{eqn:jthsummand1})
of the locus where $\|(z,\theta)\|\ge C\, k^{\epsilon-1/2}$ is $O\left(k^{-\infty}\right)$.
Adopting the rescaling $z\mapsto z/\sqrt{k}$, $\theta\mapsto \theta/\sqrt{k}$ we can rewrite (\ref{eqn:jthsummand1})
in the following form:
\begin{eqnarray}
 \label{eqn:jthsummand1rescaled}
\Pi_{k\,\boldsymbol{\nu}}(x,x)_{j1}   
&\sim& \frac{k^{1/2}\,\nu}{2\pi} \cdot 
e^{-\imath k \nu\cdot \vartheta_j}\cdot \int_{\mathbb{C}}\,\mathrm{d}V_{\mathbb{C}}(z)\,\big[I_k(x;z)\big],
\end{eqnarray}
where
\begin{equation}
 \label{eqn:Ikztheta}
I_{jk}(x;z):=\int_{-\infty}^{+\infty}\mathrm{d}\theta
\,\int_0^{+\infty}\mathrm{d}u\left[
e^{\imath\,k\,\Psi_k(x;u,\theta,z)}\cdot
A_{jk}(x;u,\theta,z)\right];
\end{equation}
in (\ref{eqn:Ikztheta}) we have set
\begin{eqnarray}
 \label{eqn:defn di Psik1}
\lefteqn{\imath\,k\,\Psi_k(x;u,\theta,z) := 
\imath\,\sqrt{k}\, \left[ u\cdot \Big\langle \Phi_G (m_x), \mathrm{Ad}_{h_{m_x}}\big(\gamma_{j1} (z,\theta)\big)\Big\rangle
-\theta\cdot \nu\right]}\nonumber\\
&&+\imath\,u\cdot \Big\langle \Phi_G (m_x), \mathrm{Ad}_{h_{m_x}}\big(\gamma_{j2} (z,\theta)\big)\Big\rangle
-\frac{u}{2}\cdot \big\|\mathrm{Ad}_{h_{m_x}}\big(\gamma_{j1} (z,\theta)\big)_X(x)\big\|^2 \nonumber\\
&&+k\, R_3\left(\frac{z}{\sqrt{k}}, \frac{\theta}{\sqrt{k}}\right),
\end{eqnarray}
\begin{eqnarray}
 \label{eqn:defn di Ak}
A_{jk}(x;u,\theta,z)
& := &s_{j1} \left(\widetilde{\mu}_{h(z/\sqrt{k}) E(- \vartheta \beta/\sqrt{k}) t_j^{-1} h(z/\sqrt{k})^{-1}} (x) ,x,k\,u\right)
\cdot \mathcal{V}_{G/T}\left(\frac{z}{\sqrt{k}}\right)\nonumber\\
&&\cdot \varrho \left(k^{-\epsilon}\, (z,\theta)\right) \cdot
\left(e^{\imath\,(\vartheta_j+\theta/\sqrt{k})}-e^{-\imath\,(\vartheta_j+\theta/\sqrt{k})}\right),
\end{eqnarray}
with $\varrho$ an appropriate bump function. Integration in $(z,\theta)$ in (\ref{eqn:jthsummand1rescaled}) 
is over a ball of radius 
$O\left(k^\epsilon\right)$ centered at the origin.

Let us derive an asymptotic expansion for (\ref{eqn:Ikztheta}).

\begin{prop}
 \label{prop:asymptotic expansion for Ikztheta}
As $k\rightarrow +\infty$, we have
\begin{eqnarray}
 \label{eqn:Ikexpanded}
I_k(x;z)&\sim& \left(\frac{k\,u_0}{\pi}\right)^d\cdot \frac{\pi}{\sqrt{k}}\cdot \frac{2\,\imath \,\sin(\vartheta_j)}{\lambda (m_x)}
\cdot e^{-\frac{u_0}{2}\,Z^t\,A(x;j)\,Z}\nonumber\\
&&\cdot \left[1+\sum_{l\ge 1} k^{-l/2}\,R_{jl}(m_x;Z)\right],
\end{eqnarray}
where $R_{jl}(m_x;Z)$ is polynomial in $Z$ of degree $\le 3l$ and parity $(-1)^l$. 
\end{prop}

\begin{proof}
[Proof of Proposition \ref{prop:asymptotic expansion for Ikztheta}]

The second summand in $\gamma_1(z,\theta)$ in (\ref{eqn:defn di gamma1e2}) is anti-diagonal. Therefore,
\begin{eqnarray}
 \label{eqn:phase manipulation}
\lefteqn{\Big\langle \Phi_G (m_x), \mathrm{Ad}_{h_{m_x}}\big(\gamma_{j1} (z,\theta)\Big\rangle}\\
&=&\Big\langle \mathrm{Ad}_{h_{m_x}^{-1}}\big(\Phi_G (m_x)\big), \gamma_{j1} (z,\theta)\Big\rangle = 
\Big\langle  \lambda (m_x)\, \beta, \theta\,\beta\Big\rangle =2\,\lambda(m_x) \,\theta.
                                          \nonumber
\end{eqnarray}

Thus we have
\begin{equation}
 \label{eqn:defn Ik(z)}
I_{jk}(x;z)=\int_{-\infty}^{+\infty}\mathrm{d}\theta
\,\int_0^{+\infty}\mathrm{d}u\left[
e^{\imath\,\sqrt{k}\,\Psi_x(u,\theta)}\cdot
\mathcal{A}_{jk}(x;u,\theta,z)\right],
\end{equation}
where now
\begin{eqnarray}
 \label{eqn:defn di PSI e A}
\Psi_x(u,\theta) &:=&\theta\cdot \big[ 2\,\lambda(m_x) \cdot u- \nu \big]\\
\mathcal{A}_{jk}(x;u,\theta,z) &:= &   e^{\mathcal{E}(u,\theta,z)}
\cdot e^{k\, R_3\left(\frac{z}{\sqrt{k}}, \frac{\theta}{\sqrt{k}}\right)}\cdot  A_{jk}(x;u,\theta,z),\nonumber
\end{eqnarray}
with
\begin{eqnarray}
\label{eqn:oscillatory phse exponent}
\mathcal{E}(u,\theta,z)& :=& \imath\,u\cdot \Big\langle \Phi_G (m_x), \mathrm{Ad}_{h_{m_x}}\big(\gamma_2 (z,\theta)\big)\Big\rangle\\
&&-\frac{u}{2}\cdot \big\|\mathrm{Ad}_{h_{m_x}}\big(\gamma_1 (z,\theta)\big)_X(x)\big\|^2 .\nonumber
\end{eqnarray}
It follows from (\ref{eqn:defn di gamma1e2}) and (\ref{eqn:oscillatory phse exponent}),
that $\mathcal{E}(u,\theta,z)$ is homogeneous of degree $2$ in $\big(\Re(z), \Im(z), \theta \big)$;
furthermore, since $\widetilde{\mu}$ is locally free at $x$, 
on the support of the integrand we have
\begin{eqnarray}
\Re \big(\mathcal{E}_k(z,\theta)\big)
& \le &
- D' \cdot \left( |z|^2 + |\theta|^2\right)
\end{eqnarray}
for some positive constant $D>0$.

Noting that $\sin(\vartheta_j) \neq 0$ as $g_j\not\in Z_x$, we have
\begin{eqnarray}
\label{eqn:Delta expanded}
e^{\imath\,(\vartheta_j+\theta/\sqrt{k})}-e^{-\imath\,(\vartheta_j+\theta/\sqrt{k})} 
= 2\,\imath \,\sin(\vartheta_j) \cdot \left[1+\sum_{l\ge 1} k^{-l/2}\,a_{jl}\cdot \theta^l\right]
\end{eqnarray}
for certain $a_{jl}\in \mathbb{R}$.
Similarly, 
\begin{eqnarray}
 \label{eqn:kR_3 expanded}
\mathcal{V}_{G/T}(z)\cdot e^{k\, R_3\left(\frac{z}{\sqrt{k}}, \frac{\theta}{\sqrt{k}}\right)}
=D_{G/T}+\sum_{r\ge 1}\,k^{-r/2}\cdot P_{r}(z,\theta)
\end{eqnarray}
where $P_{r}(z,\theta)$ is a polynomialin $(\Re(z),\Im(z),\theta)$ of degree $\le 3\,r$ and
parity $(-1)^r$ (possibly also depending on $j$).
Pairing (\ref{eqn:Delta expanded}) and (\ref{eqn:kR_3 expanded}) we conclude that
\begin{eqnarray}
 \label{eqn:Ak expanded}
\mathcal{A}_{jk}(x;u,\theta,z) &\sim&  2\,\imath \,\sin(\vartheta_j) \cdot \left(\frac{k\,u}{\pi}\right)^d\cdot 
e^{\mathcal{E}(u,\theta,z)}\nonumber\\
&&\cdot \left[1+\sum_{r\ge 1}k^{-r/2}\,P_{jr}(z,\theta) \right],
\end{eqnarray}
where $P_{jr}(z,\theta)$ is again a polynomial in $(\Re(z),\Im(z),\theta)$ of degree $\le 3\,r$ and
parity $(-1)^r$.


The following is straighforward.

\begin{lem}
 \label{lem:critical point Psi}
$\Psi_x$ has a unique critical point, which is non-degenerate and 
given by $(u_0,\theta_0)=\big(\nu/\big(2\,\lambda(m_x)\big), 0\big)$;
we have $\Psi_x(u_0,\theta_0)=0$.
The Hessian matrix has determinant $-4\,\lambda(m_x)^2$ and vanishing signature. 
\end{lem}

We can apply the Stationary Phase Lemma to determine the asymptotic expansion of
(\ref{eqn:defn Ik(z)}). In view of (\ref{eqn:defn di gamma1e2}), by a few computations we get 
\begin{eqnarray}
 \label{eqn:gamma z0}
\gamma_{j1}(z,0)&=& 
\imath\,
\begin{pmatrix}
 0 & \left(e^{-2\imath\,\vartheta_j}-1\right)\cdot z\\
\left(e^{2\imath\,\vartheta_j}-1\right)\cdot \overline{z}&0
\end{pmatrix},\\
\gamma_{j2}(z,0) &=& - |z|^2\cdot \sin(2\,\vartheta_j) \,\beta.\nonumber
\end{eqnarray}
If $z=a+\imath\,b$ with $a,b\in \mathbb{R}$, let 
$Z=
\begin{pmatrix}
 a&
b
\end{pmatrix}^t\in \mathbb{R}^2
$
be the corresponding vector; thus $|z|=\|Z\|$. Then 
$$
\big\|\mathrm{Ad}_{h_{m_x}}\big(\gamma_{j1} (z,0)\big)_X(x)\big\|^2
= \frac{1}{2}\cdot  Z^t\,C(x;j)\,Z
$$
where $C(x;j)$ is as in Definition \ref{defn:definition of Bxj}. From (\ref{eqn:oscillatory phse exponent})
and (\ref{eqn:gamma z0}) we conclude 
\begin{eqnarray*}
\mathcal{E}(u_0,0,z)
&=&-\frac{u_0}{2}\,Z^t\,B(x;j)\,Z,
\end{eqnarray*}
where 
$B(x;j)$ is also as in Definition \ref{defn:definition of Bxj}.

Finally, by (\ref{eqn:Ak expanded}) the general term of the asymptotic expansion is the evaluation
at the critical point of an expression of the form
\begin{equation}
 \label{eqn:general term asymptotic}
C_{r,l}\cdot k^{d-s/2}\,\left(\frac{\partial^2}{\partial \theta\partial u}\right)^s
\left(k^{-r/2}\,P_{jr}(z,\theta)\, e^{\mathcal{E}(u,\theta,z)}\right),
\end{equation}
for some constant $C_{r,l}\in \mathbb{C}$. Writing $\mathcal{E}(u,\theta,z)=u\cdot (H_{2,0}+H_{1,1}+H_{0,2})$,
where $H_{a,b}$ is bihomogeneous of bidegree $(a,b)$ in $(\theta,Z)$, as in Claims \ref{claim:first case} and \ref{claim:second case}
we conclude by an inductive argument that
(\ref{eqn:general term asymptotic}) has the form $k^{d-(s+r)/2}\,Q_{r,s}(\theta,Z)\, e^{\mathcal{E}(u,\theta,z)}$,
where $Q_{r,s}$ is a polynomial of degree $\le 3\,(r+s)$, and parity $(-1)^{r+s}$. The proof of Lemma 
\ref{lem:critical point Psi} is complete.
\end{proof}


Let us insert (\ref{eqn:Ikexpanded}) 
in (\ref{eqn:jthsummand1rescaled}), and remark that by parity odd polynomials do not contribute to the integral
over $\mathbb{C}$. Hence after integration the half-integer powers of $k$ drop out and we obtain  
\begin{eqnarray}
 \label{eqn:jthsummand1rescaled2}
\Pi_{k\,\boldsymbol{\nu}}(x,x)_{j1}   
&\sim& \nu \cdot 
e^{-\imath k \nu\cdot \vartheta_j}\cdot \left(\frac{k\,u_0}{\pi}\right)^d\cdot  
\frac{\imath \,\sin(\vartheta_j)}{\lambda (m_x)}
\cdot \frac{2\pi}{u_0\cdot\sqrt{\det \big(B(x;j)\big)}}\nonumber\\
&&\cdot \left[D_{G/T}+\sum_{l\ge 1} k^{-l}\,P_{jl}(m_x)\right]\nonumber\\
&=&4\pi\cdot \frac{\imath\,\sin(\vartheta_j)\cdot e^{-\imath k \nu\cdot \vartheta_j}}{\sqrt{\det \big(B(x;j)\big)}}
\cdot \left(\frac{\nu\,k}{2\,\pi\cdot \lambda(m_x)}\right)^d\nonumber\\
&&\cdot \left[D_{G/T}+\sum_{l\ge 1} k^{-l}\,P_{jl}(m_x)\right].
\end{eqnarray}

Let us remark that $g_j\neq g_j^{-1}$ since $g_j\neq \pm I_2$; 
summing the contributions (\ref{eqn:jthsummand1rescaled2}) 
corresponding to $g_j$ and $g_{j+a_x}=g_j^{-1}$, we obtain 
\begin{eqnarray}
 \label{eqn:jthsummand1rescaledsummed}
\Pi_{k\,\boldsymbol{\nu}}(x,x)_{j1} +  \Pi_{k\,\boldsymbol{\nu}}(x,x)_{j'1}  
&=&8\pi\cdot \Re\left(\frac{\imath\,\sin(\vartheta_j)\cdot e^{-\imath k \nu\cdot \vartheta_j}}{\sqrt{\det \big(B(x;j)\big)}}\right)
\cdot \left(\frac{\nu\,k}{2\,\pi\cdot \lambda(m_x)}\right)^d\nonumber\\
&&\cdot \left[D_{G/T}+\sum_{l\ge 1} k^{-l/2}\,R_{jl}(m_x)\right].
\end{eqnarray}

To deal with $\Pi_{k\,\boldsymbol{\nu}}(x,x)_{j2}$, in view of 
(\ref{eqn:inverse imagepj})
we need only go over the previous computations replacing $h_{m_x}$ with $k_{m_x}$, and $t_j$ with $t_j^{-1}$.  
In the analogue of (\ref{eqn:defn Ik(z)}), in place of the phase $\Psi_x$ in (\ref{eqn:defn di PSI e A}), we obtain 
$$
\Psi'_x(u,\theta) =-\theta\cdot \big[ 2\,\lambda(m_x) \cdot u+ \nu \big],
$$
so that $\partial_{\vartheta}\Psi'_x(u,\theta) =- \big[ 2\,\lambda(m_x) \cdot u+ \nu \big]\le -\nu$
on the domain of integration. Thus $\Pi_{k\,\boldsymbol{\nu}}(x,x)_{j2}=O\left(k^{-\infty}\right)$.

The proof of Proposition \ref{prop:g_jnotinZx} is complete.
\end{proof}

\end{proof}

\begin{proof}
 [Proof of Corollary \ref{cor:dimension estimate}]
The function $x\mapsto \Pi_{k\,\boldsymbol{\nu}}(x,x)$
is $S^1$-invariant, hence it may be interpreted as a
$\mathcal{C}^\infty$ function on $M$ (pulled back to $X$).
On the other hand, Theorems \ref{thm:G_xnotinZx} and \ref{thm:rescaled asymptotics}
imply that $x\mapsto (\pi/k)^d\cdot \Pi_{k\,\boldsymbol{\nu}}(x,x)$ is bounded,
and that, with the previous interpretation, it converges almost everywhere to the integrand on the right hand side of
(\ref{eqn:dimension estimate}). The claim follows by the dominated convergence Theorem, in view of the choice
of volume form on $X$.
\end{proof}


\begin{thebibliography}{Dillo99}
 
\bibitem[A]{apple} D. Applebaum, 
{\em Probability on compact Lie groups}, With a foreword by Herbert Heyer. Probability Theory and Stochastic Modelling, \textit{70}. 
Springer, Cham, 2014
 
 
\bibitem[BSZ]{bsz} P. Bleher, B. Shiffman, S. Zelditch, {\em
Universality and scaling of correlations between zeros on complex
manifolds}, Invent. Math. \textbf{142} (2000), 351--395 
 
 
 \bibitem[BG]{boutet-guillemin} L. Boutet de Monvel, V. Guillemin,
{\em The spectral theory of Toeplitz operators},
Annals of Mathematics Studies, \textbf{99} (1981),
Princeton University Press, Princeton, NJ;
University of Tokyo Press, Tokyo
 
\bibitem[BS]{bs} L. Boutet de Monvel, J. Sj\"ostrand,
{\em Sur la singularit\'e des noyaux de Bergman et de Szeg\"o},
Ast\'erisque \textbf{34-35} (1976), 123--164 
 
 
\bibitem[Cm]{camosso}
S. Camosso, {\em
Scaling asymptotics of Szeg\"{o} kernels under commuting Hamiltonian actions}, (English summary) 
Ann. Mat. Pura Appl. (4) \textbf{195} (2016), no. 6, 2027--2059
 
 \bibitem[Ch]{charles} L. Charles,
{\em Quantization of compact symplectic manifolds}, 
J. Geom. Anal. \textbf{26} (2016), no. 4, 2664–2710
 
 
\bibitem[GP]{gp} A. Galasso, R. Paoletti, {\em  Equivariant Asymptotics of Szeg\"{o} kernels under Hamiltonian $U(2)$ actions},   	
arXiv:1802.05644 [math.SG]
 
\bibitem[GS]{guillemin-sternberg hq} V. Guillemin, S. Sternberg, 
{\em Homogeneous quantization and multiplicities of group representations}, J. Funct. Anal. \textbf{47} (1982), no. 3, 344--380
 
\bibitem[Ko]{k}
B. Kostant, {\em Quantization and unitary representations. 
I. Prequantization}, Lectures in modern analysis and applications, III, pp. 87--208. 
Lecture Notes in Math., Vol. \textbf{170}, 
Springer, Berlin, 1970
 
 
 \bibitem[MM]{mm}
X. Ma, G. Marinescu, {\em Berezin-Toeplitz quantization and its kernel expansion. Geometry and quantization}, 
125–166, Trav. Math., \textbf{19}, Univ. Luxemb., Luxembourg, (2011)
 
 
\bibitem[MZ]{mz}
X. Ma, W. Zhang, {\em Bergman kernels and symplectic reductions}, Ast\'{e}risque, no. \textbf{318}, SMF (2008)  
 
 \bibitem[P1]{pao-jsg0}
R. Paoletti, {\em Scaling limits for equivariant Szeg\"{o} kernels}, J. Symplectic Geom. \textbf{6} (2008), no. 1, 9–32, and
\textsl{Corrigendum}, J. Symplectic Geom. \textbf{11} (2013), no. 2, 317–318

 
 \bibitem[P2]{pao-IJM} R. Paoletti, {\em Asymptotics of Szeg\"{o} kernels under Hamiltonian
torus actions}, Israel Journal of Mathematics \textbf{191} (2012), no. 1, 363--403
DOI: 10.1007/s11856-011-0212-4
 
 
\bibitem[P3]{pao-loa} R. Paoletti, {\em Lower-order asymptotics for Szeg\"{o} and Toeplitz kernels under Hamiltonian circle actions},
Recent advances in algebraic geometry, 321--369, 
London Math. Soc. Lecture Note Ser., \textbf{417}, Cambridge Univ. Press, Cambridge, 2015
 
\bibitem[RT]{rt} 
M. Ruzhansky, V. Turunen, {\em Pseudo-differential operators and symmetries. 
Background analysis and advanced topics}. Pseudo-Differential Operators. Theory and Applications, \textbf{2}. 
Birkh\"{a}user Verlag, Basel, 2010
 
 
 
\bibitem[Sch]{schlichenmaier} M. Schlichenmaier, 
{\em  Berezin-Toeplitz quantization for compact K\"{a}hler manifolds. An introduction}, 
Geometry and quantization, 97–124,
Trav. Math., \textbf{19}, Univ. Luxemb., Luxembourg, 2011
 
 
\bibitem[SZ]{sz} B. Shiffman, S. Zelditch, {\em Asymptotics of almost
holomorphic sections of ample line bundles on symplectic
manifolds}, J. Reine Angew. Math. {\bf 544} (2002), 181--222
 
 
 
\bibitem[V]{var} V. S. 
Varadarajan, {\em An introduction to harmonic analysis on semisimple Lie groups}, 
Corrected reprint of the 1989 original. Cambridge Studies in Advanced Mathematics, \textbf{16}. 
Cambridge University Press, Cambridge, 1999. x+316 pp. ISBN: 0-521-34156-6 
 
 \bibitem[Z]{z} S. Zelditch, {\em Szeg\"o kernels and a theorem of Tian},
Int. Math. Res. Not. {\bf 6} (1998), 317--331
\end{thebibliography}
\end{document}